%% file: main.tex
\documentclass{article}
\usepackage[a4paper]{geometry}
\pdfoutput=1

\usepackage[dvipsnames]{xcolor}
\usepackage[pdfencoding=unicode, hidelinks]{hyperref}
\hypersetup{
	colorlinks,
	breaklinks=true,
	linkcolor={red},
	citecolor={blue},
}
\usepackage{etoolbox}
\cslet{blx@noerroretextools}\empty %
\usepackage{mathtools}
\usepackage{mathrsfs}
\usepackage{amsthm}
\usepackage[nameinlink]{cleveref}
\usepackage{amsfonts}
\usepackage{amssymb}
\usepackage{autonum}
\usepackage[maxnames=50]{biblatex}
\usepackage[utf8]{inputenc}
\usepackage[shortlabels]{enumitem}
\usepackage[final]{microtype}
\usepackage{bm}
\usepackage{csquotes}
\usepackage[english]{babel}
\usepackage{orcidlink}
\usepackage[symbol]{footmisc}

\theoremstyle{plain}
\newtheorem{teo}{Theorem}[section]
\theoremstyle{plain}
\newtheorem{cor}[teo]{Corollary}
\theoremstyle{plain}
\newtheorem{lemma}[teo]{Lemma}
\theoremstyle{plain}

\theoremstyle{plain}
\newtheorem{prop}[teo]{Proposition}
\theoremstyle{definition}
\newtheorem{remark}[teo]{Remark}
\theoremstyle{definition}
\newtheorem{example}[teo]{Example}

\renewcommand{\r}{\mathbb{R}}
\renewcommand{\c}{\mathbb{C}}

\newcommand{\n}{\mathbb{N}}

\newcommand{\dint}{\displaystyle\int}

\newcommand{\todeb}{\rightharpoonup}

\newcommand{\dpartial}[2]{\dfrac{\partial#1}{\partial#2}}
\newcommand{\restr}[2]{\left.\kern-\nulldelimiterspace#1\vphantom{\big|}\right|_{#2}}

\DeclareMathOperator{\vspan}{span}

\DeclareMathOperator{\Id}{Id}

\DeclareMathOperator{\SO}{SO}
\DeclareMathOperator{\Ort}{O}
\DeclareMathOperator{\skw}{skew}
\DeclareMathOperator{\sym}{sym}
\DeclareMathOperator{\dist}{dist}
\DeclareMathOperator{\diag}{diag}
\DeclareMathOperator{\deter}{det}
\DeclareMathOperator{\Tr}{Tr}

\DeclareMathOperator{\rank}{rank}
\DeclareMathOperator*{\argmax}{argmax}

\DeclareMathOperator{\KL}{K}
\DeclareMathOperator{\VK}{VK}
\DeclareMathOperator{\LVK}{LVK}
\DeclareMathOperator{\CVK}{CVK}
\DeclareMathOperator{\cof}{cof}

\Crefname{lemma}{lemma}{lemmas}
\Crefname{lemma}{Lemma}{Lemmas}
\Crefname{prop}{proposition}{proposition}
\Crefname{prop}{Proposition}{Propositions}
\Crefname{cor}{corollary}{corollaries}
\Crefname{cor}{Corollary}{Corollaries}
\Crefname{remark}{remark}{remarks}
\Crefname{remark}{Remark}{Remarks}
\Crefname{teo}{theorem}{theorems}
\Crefname{teo}{Theorem}{Theorems}
\Crefname{section}{Section}{Sections}

\title{On the hierarchy of plate models for a singularly perturbed multi-well nonlinear elastic energy}

\author{Edoardo Giovanni Tolotti\footnote{Dipartimento di Matematica, Universit\`a di Pavia, Italy. \\ \indent Email address: edoardogiovann.tolotti01@universitadipavia.it} \hspace{.5mm} \orcidlink{0009-0004-3279-245X}}
\date{}

\begin{document}

\maketitle
\let\thefootnote\relax\footnotetext{This is a post-peer-review, pre-copyedit version of an article published in Journal of Nonlinear Science. The final authenticated version is available in Open Access at \url{https://doi.org/10.1007/s00332-025-10174-3}}
\thispagestyle{empty}

\input{abstract}

\input{introduction}

\input{notation}

\input{preliminary}

\input{compactness}

\input{gamma-convergence}

\input{forces}

\appendix

\input{isometries}

\input{acknowledgments}

\printbibliography

\end{document}

%% file: abstract.tex
\section*{Abstract}

In the celebrated work of Friesecke, James and M\"uller '06 the authors derive a hierarchy of models for plates by carefully analyzing the $\Gamma$-convergence of the rescaled nonlinear elastic energy. The key ingredient of their proofs is the rigidity estimate proved in an earlier work of theirs. Here we consider the case in which the elastic energy has a multi-well structure: this type of functional arises, for example, in the study of solid-solid phase transitions. Since the rigidity estimate fails in the case of compatible wells, we follow Alicandro, Dal Maso, Lazzaroni and Palombaro '18 and add a regularization term to the energy that penalizes jumps from one well to another, leading to good compactness properties. In this setting we recover the full hierarchy of plate models with an explicit dependence on the wells. Finally, we study the convergence of energy minimizers with suitable external forces and full Neumann boundary conditions. To do so, we adapt the definition of optimal rotations introduced by Maor, Mora '21.

\vspace{.2cm}

\noindent\textbf{Keywords}: Dimension reduction $\cdot$ Thin plates $\cdot$ Nonlinear elasticity $\cdot$ $\Gamma$-convergence

\noindent\textbf{AMS Classification}: 74K20 $\cdot$ 74B20 $\cdot$ 74G65 $\cdot$ 49J45

%% file: introduction.tex
\section*{Introduction}

A large variety of elastic models for thin plates, i.e., slender bodies whose reference configuration has no intrinsic curvature, have been proposed in the mechanical literature (see for example \cite{CIARLET1997, ANTMAN2005, LOVE1927, VONKARMAN1907}). The objective of this work is to rigorously derive some of these models as $\Gamma$-limits of three-dimensional nonlinear elasticity. We focus on elastic energy densities that are minimized on a finite number of copies of $\SO(3)$.
This setting is relevant, for example, in the modeling of solid-solid phase transitions (see \cite{BALL1987}).

To be more precise, we consider a thin cylinder $\Omega_h = S \times (-\frac{h}{2}, \frac{h}{2})$, where the mid-plane $S$ is a subset of $\r^2$ and $h$ is the small thickness of the plate. Its elastic energy takes the form
\[
	E_h(w_h) = \int_{\Omega_h} W(\nabla w_h) \, dx\,,
\] 
where $w_h: \Omega_h \to \r^3$ is a deformation of $\Omega_h$ and
$W$ is the elastic energy density. We suppose that $W$ is minimized on the set $K = \cup_{i=1}^l \SO(3)U_i$, where $U_1, \dots, U_l$ are symmetric and positive-definite matrices. In this work, we study the $\Gamma$-limit of $\frac{1}{h^{\alpha + 1}}E_h$, as $h \to 0$, for $\alpha \geq 2$. Roughly speaking, this corresponds to the analysis of deformations whose energy per unit volume $V_h = h^{-1}E_h$ scales like $h^{\alpha}$.

This problem has been extensively studied in the mathematical literature.
The regime $\alpha = 0$, i.e., $V_h \sim 1$, leads to the so-called membrane theory, rigorously derived for the first time by Le Dret and Raoult in \cite{LEDRET1995} (see also \cite{FREDDI2004, HANZAAFSA2008}). 
Here the hypotheses on $ W $ are compatible with a multi-well structure.
The scaling $0 < \alpha < 2$ is largely unexplored and only partial results are available (see \cite{CONTI2007}). In the single-well setting (that is, $l = 1$ and $U_1 = \Id$), the cases $V_h \sim h^\alpha$, where $\alpha \geq 2$, have been studied by Friesecke, James, and M\"uller in \cite{FRIESECKE2002, FRIESECKE2006}. When $\alpha = 2$, they retrieve the so-called Kirchhoff-Love theory, first proposed by Kirchhoff in \cite{KIRCHOFF1850}. When $\alpha = 4$, they obtain the Von K\'arm\'an model introduced in \cite{VONKARMAN1907}, whereas in the case $2 < \alpha < 4$ and $\alpha > 4$ they derive the so-called constrained and linearized Von K\'arm\'an model, respectively. The key ingredient of all the latter results is the celebrated rigidity estimate proved by the same authors in \cite{FRIESECKE2006}. Indeed, a quantitative understanding of the closeness of deformations gradients to the set of rotations is needed in order to rigorously linearize the energy density near the identity.

Various extensions of the rigidity estimate are available in the multi-well case. For two strongly incompatible wells, the rigidity estimate has been proved by Chaudhuri and M\"uller in \cite{CHAUDHURI2004} (and later with a different proof by De Lellis and Székelyhidi in \cite{DELELLIS2006}).
In this setting the same authors also investigated the scaling $\alpha = 1$ in \cite{CHAUDHURI2006}. A further generalization is given by Chermisi and Conti in \cite{CHERMISI2010}, for a finite number of well-separated wells. However, it is well-known that the rigidity estimate fails without separation assumptions on the wells. Indeed, if two of the wells are rank-one connected, it is easy to construct continuous deformations with zero elastic energy, whose gradient oscillates between the two wells. In the absence of a rigidity estimate, energy bounds fail to provide the necessary compactness for deformations. 

One possible way to overcome this issue is to singularly perturb the elastic energy by adding a higher order term of the form
\[
	h^{-1}\eta^p(h) \int_{\Omega_h} |\nabla^2 w_h|^p \, dx\,,
\]
where $\eta(h) \to 0$ as $h \to 0$ and $p > 1$ is a suitable exponent. This is a classical way of selecting preferred configurations in the Van der Waals--Cahn--Hilliard theory. Indeed, this additional term introduces a competition in the minimization problem: the free energy favors deformations with gradient in $ K $, while the perturbation penalizes transitions between the wells.

Various analyses have been carried out in this setting. In the membrane scaling $\alpha = 0$, a full description of the $\Gamma$-limit of the perturbed energy with $p = 2$ is obtained in \cite{FONSECA2006, DALMASO2010}.
As in the non-perturbed case, the hypotheses on the energy density allow for $W$ to have a multi-well structure. The expression of the $ \Gamma $-limit depends on the behavior of $ \eta(h) / h $. We also recall the work by Shu \cite{SHU2000}, where an additional homogenization parameter is taken into account.

In our work we focus on the regimes $\alpha \geq 2$, and we do not assume any hypotheses of well-separation or connectedness of the wells.
We follow ideas from Alicandro, Dal Maso, Lazzaroni, and Palombaro in \cite{ALICANDRO2018}, where the linearization of a multi-well elastic energy is studied.
We show that the perturbation coefficient $ \eta(h) $ can be chosen in such a way that, at the limit, deformation gradients are forced to fall into a single well and at the same time the (rescaled) perturbation term becomes negligible.
Under these scaling assumptions on $ \eta(h) $, in \Cref{section:compactness} we show that the $L^2$-norm of the distance of the deformation gradient from a specific well can be bounded by a suitable power of the perturbed elastic energy.
In particular, once such a well has been identified, the usual rigidity estimate can be applied to deduce compactness.

The hierarchy of plate models that we derive is similar to the one found in \cite{FRIESECKE2002,FRIESECKE2006} for the single-well case. A key difference with respect to these works is the dependence of the limit models on the well around which we have to linearize. For $\alpha = 2$, we retrieve the Kirchhoff model (see \Cref{teo:gamma_convergence_alpha_2}). Sequences of deformations with $V_h \sim h^2$ converge to isometric embeddings of the mid-plane into $\r^3$, for a flat metric depending on the well. The resulting model is similar in spirit to the one obtained in \cite{BHATTACHARYA2016}, where the authors consider a prestrained model and the limit deformation realizes an isometric embedding of the mid-plane for a metric depending on the prestrain. For $2 < \alpha < 4$, the $\Gamma$-limit is given by the constrained Von K\'arm\'an model, where the Von K\'arm\'an equations now depend on the well and take the form
\begin{equation}
	\nabla u^T + \nabla u + |U^{-1}_i e_3|^2 \nabla v \otimes \nabla v = 0 \,. \label{eq:constraint_introduction}
\end{equation}
Here $u$ and $v$ are the limiting in-plane and out-of-plane displacements with respect to a reference deformation depending on the well (i.e., of the form $x \mapsto U_i x$). Equation \eqref{eq:constraint_introduction} means that $u$ and $v$ satisfy a matching isometry condition up to the second order, for a well-dependent flat metric (see \Cref{subsection:von_karman}). For $\alpha = 4$ and $\alpha > 4$, we retrieve the Von K\'arm\'an and the linearized Von K\'arm\'an model, respectively. All the $\Gamma$-convergence results for $\alpha > 2$ are stated in \Cref{teo:gamma_convergence}. The Von K\'arm\'an model we obtain is similar to the one derived in \cite{RICCIOTTI2017} in the prestrained case. 

Concerning the proofs, once compactness is established, the liminf inequality can be obtained arguing as in \cite{FRIESECKE2002, FRIESECKE2006}. Since the penalty term requires additional regularity, the truncation argument used in \cite{FRIESECKE2002, FRIESECKE2006} for the construction of the recovery sequences cannot be applied and is replaced by suitable approximation results. In particular, for $ 2 \leq \alpha < 4 $ we need to assume some higher regularity for the mid-plane $ S $. 

The last part of the paper is devoted to the convergence of (quasi-)minimizers for the pure traction problem. More precisely, we consider a sequence of dead loads $f_h$ and add to the energy their work per unit volume given by
\begin{equation}
	- h^{-1}\int_{\Omega_h} f_h \cdot w_h \, dx \,.
\end{equation}
We assume $ f_{h} $ to be of order $ h^{\frac{\alpha}{2}+1} $. For $ \alpha > 2 $ we expect the limiting force term to depend only on the out-of-plane displacement (which is of order $ h^{\frac{\alpha}{2}-1} $) and not on the in-plane  displacement (which is of order  $ h^{\alpha-2} $ for $ 2 \leq \alpha<4 $, $ h^{ \frac{ \alpha  }{ 2 }  }  $ for $ \alpha  \geq  4 $). However, since the plate is not clamped, we need to understand how the presence of the force term reduces the rotation invariance of the problem. To do so, we follow the approach by Maor and Mora \cite{MAOR2021}, where the notion of optimal rotations is introduced. These are rotations preferred by the forces, that can be different from the rotations selected by the rigidity estimate, around which linearization takes place.
We show how this concept can be adapted to the dimension reduction setting. In this framework, we deduce a minimization property for the limit of (quasi-)minimizing sequences. A precise statement of these results is given in \Cref{teo:convergence_minimizers_alpha_2,teo:convergence_minimizers}.

The paper is organized as follows. In \Cref{section:notations} we introduce some notations, and we state the main results. After proving some preliminary lemmas, in \Cref{section:compactness} we prove compactness of sequences with bounded perturbed energy. In \Cref{section:gamma-convergence} we conclude the proof of the $\Gamma$-convergence results. Finally, in \Cref{section:forces} we study the minimizing sequences in presence of dead loads. The \Cref{app:appendix} collects some results on isometries construction.

%% file: notation.tex
\section{Notation and main results} \label{section:notations}

\subsection{Notation and functional setting}

Consider a set of invertible matrices $U_1, \ldots, U_l \in \r^{3 \times 3}$ with positive determinant and associate to each of them an energy well $K_i = \SO(3) U_i$. We assume the wells to be disjoint, i.e., $U_j^{-1} U_i \notin \SO(3)$ for any choice of $i \neq j$. We do not consider any restriction on the connectedness of the wells, in particular some $U_i$ and $U_j$ may be rank-one connected. By polar decomposition, we can suppose that $U_1, \dots, U_l$ are symmetric and positive definite. Let $K = \cup_{i=1}^{l}\SO(3)U_i$. We denote by $W \colon \r^{3 \times 3} \to [0, +\infty]$ the elastic energy density. We assume $W$ to be Borel measurable and to satisfy the following hypotheses:
\begin{enumerate}[start=1,label={(W\arabic*)}]
    \item $W(F) = 0 \iff F \in K$, \label{item:minimization_W}
    \item $W$ is $C^2$ in a neighborhood of $K$, \label{item:regularity_W}
    \item $W$ is frame indifferent, i.e., $W(RF) = W(F)$ for every $R \in \SO(3)$ and for every $F \in \r^{3 \times 3}$, \label{item:frame_indifference}
    \item $W(F) \geq C f_q(\dist(F, K))$ for every $F \in \r^{3 \times 3}$, \label{item:lower_bound_W}
\end{enumerate}
where $C > 0$ and $f_q(t) = t^2 \land t^q$ and $q \in [0,2]$. Note that far away from the set $K$ the energy density may have linear or sublinear growth. We will denote by $Q_j$ the quadratic form $D^2W(U_j)$. 

For $h > 0$ we consider a thin plate $\Omega_h = S \times (-\frac{h}{2}, \frac{h}{2}) = S \times hI$, where $S \subset \r^2$ is an open, connected and bounded set with Lipschitz boundary and $I = (-\frac{1}{2}, \frac{1}{2})$. For some results we will need higher regularity of the boundary of $S$. More precisely, we will request that 

\begin{equation}\label{eq:boundary_condition}
	\begin{gathered}
		\text{there is a closed subset } \Sigma \subset \partial S \text{ with } \mathcal{H}^1(\Sigma) = 0 \text{ such that}\\
		 \text{the outer unit normal } \vec n \text{ to } S \text{ exists and is continuous on } \partial S \backslash \Sigma\,.
	\end{gathered}
\end{equation}
This property is called \textit{condition $(\ast)$} in \cite{HORNUNG2011}. We denote by $\Omega$ the rescaled plate, that is $\Omega = \Omega_1$.

Given $\alpha \in [2, +\infty)$, set $\gamma = \frac{\alpha}{2}$. We choose $p > 1$ and $\eta \colon (0, +\infty) \to (0, +\infty)$ such that for some constant $C > 0$:
\begin{enumerate}[start=1,label={(P\arabic*)}]
    \item $\eta(h) \geq C h^{\frac{\alpha}{3}}$ for every $h > 0$, \label{item:lower_bound_penalties} 
    \item $\eta(h) h^{\gamma(1-\frac{2}{p}) - 1} \to 0$ as $h \to 0$, \label{item:convergence_penalties}
    \item $p > \frac{6}{5}$, if $q < 2$. \label{item:bound_p}
\end{enumerate}
Conditions \ref{item:lower_bound_penalties} and \ref{item:bound_p} ensure that the penalty term is strong enough to provide suitable compactness estimates (see \Cref{prop:compactness}) whereas condition \ref{item:convergence_penalties} guarantees that the penalty term is negligible at the limit.
Note that \ref{item:lower_bound_penalties}--\ref{item:bound_p} are compatible, since for every $ \alpha \geq 2 $ we have $ 1-\gamma(1- \frac{p}{2}) < \frac{\alpha}{3} $ for $ p $ large enough. 

The symbol $\nabla_h y$ denotes the rescaled gradient of $y$, while $\nabla_h^2 y$ is the rescaled Hessian of $y$. They are defined as follows:
\begin{align}
	(\nabla_h y)_{ij} & = \begin{cases}
		\partial_j y_i & \text{if } j \neq 3\,,\\
		\frac{1}{h} \partial_3 y_i & \text{otherwise}\,,
	\end{cases}\\
	(\nabla_h^2 y)_{ijk} & = \begin{cases}
		\partial_{jk} y_i & \text{if } j,k \neq 3\,,\\
		\frac{1}{h} \partial_{3k} y_i & \text{if } j = 3\,,k \neq 3\,,\\
		\frac{1}{h} \partial_{j3} y_i & \text{if } j \neq 3\,,k = 3\,,\\
		\frac{1}{h^2} \partial_{33} y_i & \text{if } j = 3\,,k = 3\,.\\
	\end{cases}
\end{align}
We set 
\[
    \mathcal{H}_p(\Omega; \r^3) = \{y \in W^{1,2}(\Omega; \r^3) \colon \nabla^2_h y \in L^p(\Omega; \r^{3 \times 3 \times 3})\}\,,
\]
and for $\alpha \geq 2$ we define the $\alpha$-rescaled energy $E^\alpha_h \colon W^{1,2}(\Omega, \r^3) \to \r$ as follows:
\begin{equation}\label{eq:rescaled_energy}
    E^\alpha_h(y) = \begin{cases}
        \dfrac{1}{h^\alpha}\dint_\Omega W(\nabla_h y) \, dx + \dfrac{\eta^p(h)}{h^\alpha}\dint_\Omega |\nabla_h^2 y|^p \, dx \, \, & \text{if} \, y \in \mathcal{H}_p(\Omega; \r^3) \,,\\
        + \infty \,\, & \text{otherwise} \,.
    \end{cases}
\end{equation}

We explain now some useful matrix notation that we will use throughout the paper. Given a general matrix $A \in \r^{3 \times 3}$ we will write $A'$ to denote the $2 \times 2$ submatrix obtained removing the third row and column. Similarly, we will often write $x'$ in place of $(x_1, x_2)$ and $\nabla'$ instead of $\nabla_{x_1, x_2}$. We will use the super(sub)scripts to denote submatrices of $A$ in the following way: every missing subscript index is a removed row while every missing superscript index is a removed column. For example, $A^{1,2}$ is the $3 \times 2$ submatrix given by the first two columns of $A$ while $A_{1,2}$ is the $2 \times 3$ submatrix given by the first two rows of $A$. Whenever we will sum or multiply matrices and vectors with different dimension we will imply that the smaller one is naturally embedded in the bigger space by adding zeros in the missing entries. For example, if $F \in \r^{2 \times 2}$ and $G \in \r^{3 \times 3}$ the expression $F + G$ means $\iota(F) + G$ where
\[
    \iota: \r^{2 \times 2} \hookrightarrow \r^{3 \times 3} \,, \quad  F \mapsto \begin{pmatrix}
        F & 0 \\
        0 & 0
    \end{pmatrix} \,.
\]
We will denote by $I_h$ the matrix
\begin{equation}\label{eq:definition_Ih}
	I_h = \begin{pmatrix}
	1 & 0 & 0 \\
	0 & 1 & 0 \\
	0 & 0 & h
\end{pmatrix} \,.
\end{equation}

\subsection{Main results}

The main objective of this paper is to extend the rigorous derivation of a hierarchy of plate models obtained by Friesecke, James, and M\"uller in \cite{FRIESECKE2002, FRIESECKE2006} for a single well to the case of a multiwell elastic energy. In order to do so, we need to introduce the limiting models centered at a well $K_j$. We define
\begin{equation}\label{eq:definition_Q}
    \bar Q_j \colon \r^{2 \times 2} \to \r \quad F \mapsto \min_{a \in \r^3} Q_j(U_j^{-1}(\sym(F) + a \otimes e_3 + e_3 \otimes a)) \,.
\end{equation}
This is a quadratic form that generalizes the one defined in \cite[Equation (8)]{FRIESECKE2006}, taking into account that linearization may arise around a possibly nontrivial deformation gradient $U_j$. The minimum in $\bar Q_j$ exists thanks to some coercivity property of $Q_j$ that will be proved later (see \Cref{lemma:coercivity_second_derivative}).

\subsubsection{The Kirchhoff regime \texorpdfstring{$\bm{\alpha = 2}$}{\unichar{"03B1} = 2}}

For $\alpha = 2$, we will prove the $\Gamma$-convergence of $E_h^\alpha$ to the Kirchhoff functional
\begin{equation}\label{eq:kirchoff_love}
    E^{\KL}_j(y) = \dfrac{1}{24}\displaystyle\int_S \bar Q_j(\nabla y^T \nabla \nu) \, dx \,,
\end{equation}
where $y \in W^{2,2}(S; \r^3)$ satisfies the constraint $(\nabla y)^T \nabla y = (U_j^2)'$ and $\nu$ is the unique vector such that $(\nabla y, \nu)U_j^{-1} \in \SO(3)$ a.e. in $S$ (whose existence is guaranteed by \Cref{lemma:definition_nu}). Note that $y$ is an isometric immersion for the flat metric $(U^2_j)'$. The $\Gamma$-convergence result is stated in the following theorem.

\begin{teo}[$\Gamma$-convergence ($\alpha = 2$)]\label{teo:gamma_convergence_alpha_2}
    Suppose that $S$ satisfies \eqref{eq:boundary_condition}. 
    \begin{enumerate}[(i)]
    	\item For any sequence $(y_h) \subset W^{1,2}(\Omega; \r^3)$ that satisfies $E_h^2(y_h) \leq C$ for every $ h > 0 $ there exist $y \in W^{2,2}(S; \r^3)$ and $j \in \{1, \dots, l\}$ such that
    	\begin{enumerate}[(a)]
    		\item $\nabla y^T \nabla y = (U_j^2)'$, \label{item:isometry_condition}
    		\item $\nabla_h y_h \to (\nabla y, \nu)$ in $L^2(\Omega; \r^{3 \times 3})$ (up to a subsequence), where $\nu$ is defined as in \Cref{lemma:definition_nu}. \label{item:convergence_rescaled_gradient_compactness}
    	\end{enumerate} 
  		\label{item:compactness_alpha_2}
    	\item For any sequence $(y_h) \subset W^{1,2}(\Omega; \r^3)$ as in \ref{item:compactness_alpha_2} it holds
    	\[
    		\liminf\limits_{h \to 0} E_h^2(y_h) \geq E^{\KL}_j(y)\,.
    	\]\label{item:liminf_inequality_alpha_2}
    	\item For any $y \in W^{2,2}(S; \r^3)$ such that $\nabla y^T \nabla y = (U_j^2)'$ for some index $j \in \{1, \dots, l\}$, there exists a sequence $(y_h) \subset W^{1,2}(\Omega; \r^3)$ such that \ref{item:convergence_rescaled_gradient_compactness} holds true and
    	\[
    	\lim_{h \to 0} E_h^2(y_h) = E^{\KL}_j(y) \,.
    	\]\label{item:recovery sequence_alpha_2}
    \end{enumerate}
\end{teo}

\subsubsection{The Von K\'arm\'an regimes \texorpdfstring{$\bm{\alpha > 2}$}{\unichar{"03B1} > 2}}\label{subsection:von_karman}

When $\alpha > 2$, the limiting models are expressed in terms of $u$ and $v$, the limits of the rescaled in-plane and out-of-plane averaged displacements around a well $K_j$, called $u_h$ and $v_h$, respectively. They are defined as follows:
\begin{align}
    & u_h \colon S \to \r^2 && x' \mapsto \min\left\{h^{-\gamma}, h^{2-2 \gamma}\right\}\begin{pmatrix}
         w_h(x') \cdot U_j e_1\\
         w_h(x') \cdot U_j e_2 
    \end{pmatrix}\,, \label{eq:definition_u}\\
    & v_h \colon S \to \r && x' \mapsto h^{1-\gamma} w_h(x') \cdot U_j e_3 \,, \label{eq:definition_v}
\end{align}
where we recall that $\gamma = \frac{\alpha}{2}$ and
\[
    w_h \colon S \to \r^3 \quad x' \mapsto \int_I (y_h - U_j \Id_h x)\, dx_3 \,.
\]
Observe that $u_h$ and $v_h$ are the components of the averaged displacement $w_h$ in the basis given by $\{U_j^{-1}e_1, U_j^{-1}e_2, U_j^{-1}e_3\}$. This may not look as a natural choice, since a basis of tangent vectors to the embedded midplane $U_j(S \times \{0\})$ is given by $\{U_je_1, U_je_2\}$ and the normal direction is given by $U^{-1}_j e_3$.
However, since $ \{U_{j}^{-1}e_{i} : i = 1, 2, 3\} $ is the dual basis of $ \{U_{j}e_{i} : i = 1, 2, 3 \} $, this alternative simplifies both the statement and the computations, and gives a completely equivalent result (see \Cref{remark:basis_u_v}). When $2 < \alpha < 4$, we retrieve the $\Gamma$-convergence to the constrained Von K\'arm\'an functional, namely,
\[
    E^{\CVK}_j(v) = \dfrac{1}{24} \displaystyle\int_S \bar Q_j(\nabla^2 v) \, dx \,,
\]
whenever $v \in W^{2,2}(S)$ is such that there exists $u \in W^{1,2}(S; \r^2)$ that solves
\begin{equation}
	\nabla u^T + \nabla u + |U^{-1}_j e_3|^2 \nabla v \otimes \nabla v = 0 \,. \label{eq:constraint_u_v}
\end{equation}
As in the single well case, this constraint means that $u$ and $v$ satisfy a matching isometry condition up to the second order; however, the metric now depends on the well, and it is not necessarily the Euclidean one. More precisely, if one defines for $\varepsilon > 0$
\[
	y_\varepsilon = U_j \begin{pmatrix}
		x'\\ 0
	\end{pmatrix} + \varepsilon U_j^{-1}e_3 v + \varepsilon^2 U_j^{-1} \begin{pmatrix}
		u \\ 0
	\end{pmatrix} \,,
\]
then \eqref{eq:constraint_u_v} is equivalent to $\nabla y_\varepsilon^T \nabla y_\varepsilon = (U_j^2)' + O(\varepsilon^3)$.

\begin{remark}\label{remark:determinant_zero}
    If $S$ is simply connected, equation \eqref{eq:constraint_u_v} implies in particular that $\det(\nabla^2 v) = 0$. For a proof of this result see  \cite[Proposition 9]{FRIESECKE2006}. 
\end{remark}

The case $\alpha = 4$ corresponds to the Von K\'arm\'an model, that is,
\begin{align}
    E^{\VK}_j(u, v) & =
        \dfrac{1}{24}\displaystyle\int_S \bar Q_j(\nabla^2 v) \, dx \\
        & \hphantom{=} \, \, + \dfrac{1}{8}\displaystyle\int_S \bar Q_j(\nabla u^T + \nabla u + |U^{-1}_je_3|^2 \nabla v \otimes \nabla v) \, dx\,,
\end{align}
for $v \in W^{2,2}(S)$ and $u \in W^{1,2}(S; \r^2)$.

Lastly, if $\alpha > 4$, we prove the $\Gamma$-convergence to the linearized Von K\'arm\'an model:
\[
    E^{\LVK}_j(u, v) = \dfrac{1}{24} \displaystyle\int_S \bar Q_j(\nabla^2 v) \, dx + \dfrac{1}{8}\displaystyle\int_S \bar Q_j(\nabla u^T + \nabla u) \, dx \,,
\]
for $v \in W^{2,2}(S)$ and $u \in W^{1,2}(S; \r^2)$. The following theorem summarizes these $\Gamma$-convergence results.

\begin{teo}[$\Gamma$-convergence ($\alpha > 2$)]\label{teo:gamma_convergence}
	Suppose $\alpha > 2$.
	\begin{enumerate}[(i)]
		\item For any sequence $(y_h) \subset W^{1,2}(\Omega; \r^3)$ that satisfies $E_h^\alpha(y_h) \leq C$ for every $h > 0$ there exist an index $j \in \{1, \dots, l\}$, two sequences $(\bar R_h) \subset \SO(3)$, $(c_h) \subset \r$, and two maps $v \in W^{2,2}(S)\,, u \in W^{1,2}(S; \r^2)$ such that, up to a subsequence, the following convergences hold:
		\begin{enumerate}[(a)]
			\item $u_h \todeb u$ in $W^{1,2}(S; \r^2)$, \label{item:convergence_u}
			\item $v_h \to v$ in $W^{1,2}(S)$, \label{item:convergence_v}
		\end{enumerate}
		where $u_h$ and $v_h$ are, respectively, the in-plane and out-of-plane displacements around the well $K_j$ defined as in \eqref{eq:definition_u}--\eqref{eq:definition_v} for the roto-translated deformation
		\[
		\tilde y_h = \bar R_h y_h + c_h\,.
		\] 
		Moreover, if $2 < \alpha < 4$, then
		\begin{equation} \label{eq:isometry_constraint}
				\nabla u^T + \nabla u + |U^{-1}_j e_3|^2 \nabla v \otimes \nabla v = 0 \,.
		\end{equation} \label{item:compactness}
		\item For any sequence $(y_h) \subset W^{1,2}(\Omega; \r^3)$ as in \ref{item:compactness}
		\begin{enumerate}[(a)]
			\item if $2 < \alpha < 4$, then $ \liminf\limits_{h \to 0} E_h^\alpha(y_h) \geq E_j^{\CVK}(v)$,
			\item if $\alpha = 4$, then $ \liminf\limits_{h \to 0} E_h^\alpha(y_h) \geq E_j^{\VK}(u, v)$,
			\item if $\alpha > 4$, then $ \liminf\limits_{h \to 0} E_h^\alpha(y_h) \geq E_j^{\LVK}(u, v)$.
		\end{enumerate}\label{item:liminf_inequality}
	\item Suppose that $2 < \alpha < 4$ and that $S$ is simply connected and satisfies \eqref{eq:boundary_condition}. For any choice of $j \in \{1, \dots, l\}$ and $v \in W^{2,2}(S)$ such that there exists $u \in W^{1, 2}(S; \r^2)$ solving \eqref{eq:isometry_constraint}, there exists a sequence $(y_h) \subset W^{1,2}(\Omega; \r^3)$ such that
	\begin{enumerate}[(a)]
		\item $v_h \to v$ in $W^{1, 2}(S)$, where $v_h$ is defined as in \eqref{eq:definition_v},
		\item $\lim\limits_{h \to 0} E_h^\alpha(y_h) = E^{\CVK}(v)$.
	\end{enumerate}\label{item:recovery_sequence_2_alpha_4}
	\item For any choice of $j \in \{1, \dots, l\}$, $u \in W^{1,2}(S; \r^2)$, and $v \in W^{2,2}(S)$ there exists a sequence $(y_h) \subset W^{1,2}(\Omega; \r^3)$ such that, defining $u_h \,, v_h$ as in \eqref{eq:definition_u}-\eqref{eq:definition_v},
	\begin{enumerate}[(a)]
		\item $v_h \to v$ in $W^{1,2}(S)$,
		\item $u_h \todeb u$ in $W^{1, 2}(S; \r^2)$,
		\item $\lim\limits_{h \to 0} E_h^\alpha(y_h) = E^{\VK}(u, v)$ if $\alpha = 4$, \label{item:recovery_sequence_alpha_4}
		\item $\lim\limits_{h \to 0} E_h^\alpha(y_h) = E^{\LVK}(u, v)$ if $\alpha > 4$. \label{item:recovery_sequence_alpha_greater_4}
	\end{enumerate}\label{item:recovery_sequence_alpha_greater_equal_4}
	\end{enumerate}
\end{teo}

\begin{remark}
    In the proof of \Cref{teo:gamma_convergence}--\ref{item:recovery_sequence_2_alpha_4} we cannot use the truncation argument of \cite{FRIESECKE2002,FRIESECKE2006}.
    Indeed, the penalty term in the energy requires higher regularity.
    To overcome this issue we suppose that $ S $ satisfies \eqref{eq:boundary_condition}, so that we can apply the approximation result given in \cite{HORNUNG2011}. 
\end{remark}

\subsubsection{Convergence of minimizers in the presence of dead loads}

External forces can be included in the previous analysis. We suppose $q > 1$, and we study the convergence of minimizers of the rescaled total energy
\begin{equation}\label{eq:definition_total-energy}
	J_h^\alpha \colon W^{1, 2}(\Omega; \r^3) \to \r \cup \{+ \infty\}\,, \quad J_h^\alpha(y_h) = E_h^\alpha(y_h) - \dfrac{1}{h^\alpha}\int_\Omega f_h \cdot y_h \, dx \,,
\end{equation}
where $f_h \colon S \to \r^3$ is a sequence of dead loads that satisfies
\begin{equation}\label{eq:convergence_forces}
	\dfrac{1}{h^{\gamma+1}}f_h \to f \quad \text{in} \, \, L^{q'}(S; \r^3) \,.
\end{equation}
Here $q'$ is the conjugate exponent of $q$. Note that $q' \geq 2$, so that the strong convergence of the forces holds also in $L^2(S; \r^3)$. We assume the forces to be mean-free, i.e.,
\[
	\int_S f_h \, dx = 0\,,
\]
otherwise the infimum of $J_h^\alpha$ is $-\infty$. We prove the following result.

\begin{teo}\label{teo:convergence_minimizers_alpha_2}
	Suppose that $S$ satisfies \eqref{eq:boundary_condition} and $q > 1$. Suppose that $(y_h) \subset W^{1,2}(\Omega; \r^3)$ is a sequence of deformations that are quasi-minimizers for $J_h^2$, i.e.,
	\[
		\limsup_{h \to 0} \, (J_h^2(y_h) - \inf J_h^2) =0 \,.
	\]
	Then, $E_h^2(y_h) \leq C$ for every $h > 0$ and there exist $y \in W^{2, 2}(S; \r^3)$ and $j \in \{1, \dots, l\}$ satisfying $(\nabla y)^T\nabla y = (U_j^2)'$ such that, up to subsequences, $y_h \to y$ in $W^{1, 2}(\Omega; \r^3)$ and $(j, y)$ minimizes
	\[
		J^{\KL}_j(y) = E^{\KL}_j(y) - \int_S f \cdot y \, dx \,,
	\]
	over the set
	\[
		\{(j, y) \in \{1, \dots, l\} \times W^{2, 2}(S; \r^3) \colon \nabla y^T \nabla y = (U_j^2)'\}\,.
	\]
\end{teo}

For the case $\alpha >  2$ we need to introduce some notation. We denote by $\mathcal{M}_h \subset K$ the set of maximizers of the functional
\[
F_h \colon K \to \r, \quad RU_j \mapsto \int_S f_h \cdot RU_j \begin{pmatrix}
x' \\ 0
\end{pmatrix} \, dx \,.
\]
Note that $\mathcal{M}_h$ is not empty by compactness of $K$. Similarly, we define 
\[	
\mathcal{M} = \argmax_{K} F \,,
\]
where
\[
	F(A) = \int_S f \cdot A \begin{pmatrix}
	x' \\ 0
	\end{pmatrix} \, dx \,.
\]
We define the sets $\mathcal{R}^j$ and $\mathcal{R}_h^j$ as follows
\begin{align}
	\mathcal{R}^j & = \argmax_{R \in \SO(3)} F(RU_j)\,,\\
	\mathcal{R}_h^j & = \argmax_{R \in \SO(3)} F_h(RU_j)\,.
\end{align}
By definition of $K$, there are subsets of indices $\Lambda_h, \Lambda \subset \{1, \dots, l\}$ such that $\mathcal{M}_h = \bigcup_{j \in \Lambda_h} \mathcal{R}_h^jU_j$
and $\mathcal{M} = \bigcup_{j \in \Lambda} \mathcal{R}^jU_j$.
These definitions generalize to our context the notion of optimal rotations introduced in \cite{MAOR2021}. One can prove that $\mathcal{R}_h^j$ and $\mathcal{R}^j$ are closed, connected, boundaryless, and totally geodesic submanifolds of $\SO(3)$. Indeed, the proof in \cite{MAOR2021} does not rely on the specific structure of $F$ or $F_h$, but only on their linearity. We denote by $T \mathcal{R}^j_R$ and $N \mathcal{R}^j_R$ the tangent space and the normal space to $\mathcal{R}^j$ at the point $R$, respectively. By \cite[Proposition 4.1]{MAOR2021}, we have
\begin{align}
	T\mathcal{R}^j_R &= \{RW \in \r^{3 \times 3} \colon W \in \r^{3 \times 3}_{\skw}\,, F(RW^2U_j) = 0\}\,, \\
	N\mathcal{R}^j_R &= \{RW \in \r^{3 \times 3} \colon W \in \r^{3 \times 3}_{\skw}\,, RW \perp T\mathcal{R}^j_R\} \,. 
\end{align}
Similarly, we define $T \mathcal{R}_h^j{}_R$ and $N \mathcal{R}_h^j{}_R$. We can define the projection operators $P^j$ and $P_j^h$ of $\SO(3)$ onto $\mathcal{R}^j$ and $\mathcal{R}_h^j$, respectively. These projections have to be understood with respect to the intrinsic distance of $\SO(3)$, i.e.,
\[
	\dist_{\SO(3)}(R, Q) = \min\left\{|W| \colon W \in \r^{3 \times 3}_{\skw}\,, Q = Re^W\right\} \,,
\]
and are well-defined at least in a neighborhood of $\mathcal{R}^j$ and $\mathcal{R}_h^j$, respectively. We will suppose that the forces $f_h$ are such that 
\begin{enumerate}[start=1,label={(F\arabic*)}]
	\item $\Lambda_h = \Lambda$ for $h \ll 1$, \label{item:lambda_h_equals_lambda}
	\item $\dim \mathcal{R}_h^j \to \dim \mathcal{R}^j$ for any $j \in \Lambda$. \label{item:convergence_dimensions}
\end{enumerate}
Note that in general one only has $ \Lambda_{h} \subseteq \Lambda $ and $ \limsup_{h \to 0} \dim R_{h}^{j} \leq \dim R_{h}^{j} $, as shown in the following example.
The failure of \ref{item:lambda_h_equals_lambda}--\ref{item:convergence_dimensions} may happen, for instance, when the direction along which the force acts is slightly perturbed.

\begin{example}
    Let $ \alpha > 2 $ and set $ S = (-\frac{1}{2},\frac{1}{2})^{2} $.
    Suppose that $ l =2 $ and let
    \begin{equation}
        U_{1} = \Id, \qquad U_{2} = \begin{pmatrix}
            1 & 0 & 0\\
            0 & 2 & 0\\
            0 & 0 & 1
        \end{pmatrix}.
    \end{equation}
    We consider the following sequence of forces
    \begin{equation}
        f_{h}(x') = h^{\gamma + 1} [x_{1}e_{3} + hx_{2}e_{2}] .
    \end{equation}
    Note that the sequence $ f_{h} $ is mean-free by simmetry.
    Then, with some simple computation one has that
    \begin{align}
        F_{h}(RU_{1}) & = \frac{h^{\gamma + 1}}{12}(R_{31} + h R_{22}), \\
        F_{h}(RU_{2}) & = \frac{h^{\gamma + 1}}{12}(R_{31} + 2h R_{22}) . 
    \end{align}
    It follows that $ \mathcal{R}_{h}^{1} = \mathcal{R}_{h}^{2}$ are singletons given by the matrix
    \begin{equation}
        \begin{pmatrix}
            0 & 0 & 1\\
            0 & 1 & 0\\
            1 & 0 & 0
        \end{pmatrix},
    \end{equation}
    and that  $ \Lambda_{h} = \{2\} $.
    The limit force $ f $ is given by $ f(x') = x_{1} e_{3} $.
    Hence, 
    \begin{equation}
        F_{h}(RU_{1}) = F(RU_{2}) =  \frac{1}{12}R_{31},
    \end{equation}
    Thus,
    \begin{equation}
        \mathcal{R}^{1} = \mathcal{R}^{2} = \{R \in \SO(3) : R_{31} = 1\},
    \end{equation}
    and $ \dim \mathcal{R}^{1} = \dim \mathcal{R}^{2} = 1 $.
    Moreover, $ \Lambda = \{1, 2\} $.
\end{example}

Under the assumptions \ref{item:lambda_h_equals_lambda}--\ref{item:convergence_dimensions}, the following result holds.
\begin{teo}\label{teo:convergence_minimizers}
	Let $\alpha > 2$ and $q > 1$. Assume \ref{item:lambda_h_equals_lambda}--\ref{item:convergence_dimensions}. Suppose that $(y_h) \subset W^{1,2}(\Omega; \r^3)$ is a sequence of deformations that are quasi-minimizers for $J_h^\alpha$, i.e.,
	\[
	\limsup_{h \to 0} \, (J_h^\alpha(y_h) - \inf J_h^\alpha) =0 \,.
	\]
	Then, $E_h^\alpha(y_h) \leq C$ for every $h > 0$ and there exist $(\bar R_h) \subset \SO(3)$, $(c_h) \subset \r$, $u \in W^{1, 2}(S; \r^2)$, $v \in W^{2, 2}(S)$, and $j \in \{1, \dots, l\}$ such that $j \in \Lambda$ and, up to subsequences,
	\begin{enumerate}[(i)]
		\item $\bar R_h \to \bar R$ with $\bar R^T U_j \in \mathcal{M}$,
		\item $u_h \todeb u$ in $W^{1, 2}(\Omega; \r^2)$,
		\item $v_h \to v$ in $W^{1, 2}(\Omega)$,
	\end{enumerate}
	where $u_h, v_h$ are defined as in \eqref{eq:definition_u}--\eqref{eq:definition_v} for $\tilde y_h = \bar R_h y_h + c_h$. In addition, there is $W \in \r^{3 \times 3}_{\skw}$ such that $|W| = 1$, $\bar R^TW \in N\mathcal{R}^j_{\bar R^T}$, and 
	\[
	h^{\frac{1}{2}(1-\gamma)}(\bar R_h^T - P^j_h(\bar R_h)) \to \beta \bar R^T W \,, \label{eq:convergence_difference_projections}
	\]
	for some $\beta \geq 0$. Lastly,
	\begin{enumerate}[(a)]
		\item if $2 < \alpha < 4$ and $S$ is  a simply connected set that satisfies \eqref{eq:boundary_condition}, then $(j, v, \bar R^T, \beta W)$ minimizes the functional
		\[
			J^{\CVK}_j(v, R, W) = E^{\CVK}_j(v) - \int_S f \cdot R U_j^{-1} e_3 v \, dx - F(R W^2 U_j)
		\]
		over all the admissible quadruplets
		$(j, v, R, W)$, i.e., $j \in \Lambda$ and $(v, R, W) \in \mathcal{C} \times \mathcal{R}^j \times \r^{3 \times 3}_{\skw}$ is such that $W \in N\mathcal{R}^j_{R^T}$, where $\mathcal{C}$ is the set
		\[
			\mathcal{C} = \left\{v \in W^{2, 2}(S) \colon \exists\, u \in W^{1, 2}(S; \r^2) \, \, \text{s.t.} \, \, \nabla u^T + \nabla u + |U_j^{-1}e_3| \nabla v \otimes \nabla v = 0\right\} \,,
		\]
		
		\item if $\alpha = 4$, then $(j, u, v, \bar R^T, \beta W)$ minimizes the functional
		\[
			J^{\VK}_j(u, v, R, W) = E^{\VK}_j(u, v) - \int_S f \cdot R U_j^{-1} e_3 v \, dx - F(R W^2U_j)
		\]
		over all the admissible quintuplet $(j, u, v, R, W)$, i.e., $j \in \Lambda$ and $(u, v, R, W) \in W^{1,2}(S; \r^2) \times W^{2, 2}(S) \times \mathcal{R}^j \times \r^{3 \times 3}_{\skw}$ is such that $W \in N\mathcal{R}^j_{R^T}$,
		\item if $\alpha > 4$, then $(j, u, v, \bar R^T, \beta W)$ minimizes the functional
		\[
			J^{\LVK}_j(u, v, R, W) = E^{\LVK}_j(u, v) - \int_S f \cdot R U_j^{-1} e_3 v \, dx - F(R W^2 U_j)
		\]
		over all the admissible quintuplets $(j, u, v, R, W)$, i.e., $j \in \Lambda$ and $(u, v, R, W) \in W^{1,2}(S; \r^2) \times W^{2, 2}(S) \times \mathcal{R}^j \times \r^{3 \times 3}_{\skw}$ is such that $W \in N\mathcal{R}^j_{R^T}$.\label{item:minimization_alpha_greater_4}
	\end{enumerate}

\end{teo}

\begin{remark}
    Given that the scaling of the forces is of order $h^{\gamma + 1}$, we expect the action of the load on the in-plane displacement to be negligible. This is indeed the case and the limiting forcing term acts only on the out-of-plane displacement. The additional term $F(RW^2U_j)$ can be interpreted (see also \cite{MAOR2021}) as an elastic cost of fluctuations of the reference configuration from the optimal rotations.

    From a minimization point of view, since the term $F(RW^2U_j)$ is always non-positive (see \Cref{section:forces}), it is clear that the optimal choice is $W = 0$. In particular, in \Cref{teo:convergence_minimizers} we actually have $\beta = 0$. Similarly, since
    \begin{equation}
        E_j^{\LVK}(0, v) \leq E_j^{\LVK}(u, v) \quad \forall \, u \in W^{1, 2}(S; \r^2) \quad \forall \, j = 1, \dots, l\,,
    \end{equation}
    in \Cref{teo:convergence_minimizers}--\ref{item:minimization_alpha_greater_4} we also have $u = 0$.
\end{remark}

We provide now an example of rank-one connected double-well structure for which different applied forces result in different preferred reference configurations.
\begin{example}
    Let $ S = (-\frac{1}{2},\frac{1}{2})^{2} $ and consider 
    \begin{equation}
        U_{1} = \begin{pmatrix}
            4 & 0 & 0 \\
            0 & 1 & 0\\
            0 & 0 & 1
        \end{pmatrix}, \qquad U_{2} = \begin{pmatrix}
            2 & 0 & 1\\
            0 & 1 & 0\\
            1 & 0 & 1
        \end{pmatrix}.
    \end{equation}
    For $ a $, $ b $, $ c \geq 0 $, consider the sequences of loads
    \begin{equation}
        f_{h}(x') = h^{\gamma + 1} (a x_{1}e_{1} + b x_{2} e_{2} + c x_{1}e_{3}).
    \end{equation}
    The limit force $ f(x') = ax_{1}e_{1} + bx_{2}e_{2} + cx_{1} e_{3} $ is pulling the mid-plane $ S $ along fibers parallel to $ e_{1} $ and $ e_{2} $, while twisting it in the out-of-plane direction.
    With some simple computation we get
    \begin{align}
        F(RU_{1}) & = \frac{1}{12}(4a R_{11} + bR_{22} + cR_{31}),\\
        F(RU_{2}) & = \frac{1}{12}[a(2R_{11} +  R_{13}) + bR_{22} + c(2R_{31} + R_{33})]. 
    \end{align}
    Note that, if $ a = 0 $ and $ b $, $ c > 0 $ (i.e., $ f $ is pulling the mid-plane $ S $ in the $ e_{2} $ direction only), then one has that 
    \begin{equation}
        F(RU_{1}) \leq \frac{1}{12}(b + c) < F(\bar R U_{2}) , \qquad \forall \, R \in \SO(3),
    \end{equation}
    where
    \begin{equation}
        \bar R = \begin{pmatrix}
            -\frac{\sqrt{2} }{2} & 0 & \frac{\sqrt{2}}{2}\\
            0 & 1 & 0\\
            \frac{\sqrt{2}}{2} & 0 & \frac{\sqrt{2}}{2}
        \end{pmatrix}.
    \end{equation}
    In particular, $ \Lambda = \{2\} $ and the only admissible well at the limit is $ \SO(3)U_{2} $.
    However, if $ a > 0 $ and $ b = c = 0 $, that is $ f $ is tensing $ S $ along $ e_{1} $ without twisting the mid-plane, we have
    \begin{align}
        F(RU_{2}) = \frac{a}{12}(2R_{11} + R_{13}) \leq \frac{a}{4} < \frac{a}{3} = F(U_{1}) , \qquad \forall \, R \in \SO(3) \,. 
    \end{align}
    Thus, $ \Lambda = \{1\} $ and the only admissible reference configuration is $ \SO(3)U_{1} $.
\end{example}

%% file: preliminary.tex
\section{Preliminary results}

In this section we give a few preliminary results regarding the energy $W$. These properties are well known in the case of a single well elastic energy. Firstly, we explore the invariance properties of the second derivative $Q_j$ that follow from frame indifference \ref{item:frame_indifference}. Lastly, we prove that \ref{item:lower_bound_W} implies some coercivity for $Q_j$. 

\begin{lemma}\label{lemma:simmetry_second_derivative}
    The following equality holds true for every $j = 1, \dots, l$:
    \[
        Q_j(A) = Q_j(\sym(AU_j^{-1})U_j) = Q_j(U_j^{-1}\sym(U_jA)) \quad \forall \, A \in \r^{3 \times 3}\,.
    \]
\end{lemma}

\begin{proof}
    Define $\tilde W(F) = W(FU_j)$. By the properties \ref{item:minimization_W}--\ref{item:frame_indifference} we deduce that $\tilde W$ is frame indifferent, it is $C^2$ in a neighborhood of $\SO(3)$, and it is minimized on $\SO(3)$. It is well known that any map satisfying these properties has symmetric second derivative at the identity, namely:
    \[
        D^2\tilde W(\Id)[A]^2 = D^2\tilde W(\Id)[\sym(A)]^2 \,.
    \]
    By some simple computation we get $D^2 \tilde W(\Id)[A]^2 = Q_j(AU_j)$. Hence,
    \[
        Q_j(A) = Q_j(AU_j^{-1}U_j) = Q_j(\sym(AU_j^{-1})U_j)\,.
    \]
    To conclude we observe that by definition of symmetric part and the symmetry of $U_j$ we have 
    \begin{align}
        Q_j(\sym(AU_j^{-1})U_j) & = Q_j(\sym(U_j^{-1}U_jAU_j^{-1})U_j)  \\
        & = Q_j(U_j^{-1}\sym(U_jA)) \,.
    \end{align}
\end{proof}

\begin{lemma}\label{lemma:coercivity_second_derivative}
    There exists $\lambda > 0$ such that for every $j = 1, \dots, l$ and every $A \in \r^{3 \times 3}$
    \[
        Q_j(U_j^{-1}\sym(A)) \geq \lambda |\sym(A)|^2 \,.
    \]
\end{lemma}

\begin{proof}
    By \ref{item:lower_bound_W} we have, for $\varepsilon \ll 1$
    \[
        W(U_j + \varepsilon U_j^{-1}A) \geq C \dist^2(U_j + \varepsilon U_j^{-1}A, K) = C \dist^2(U_j + \varepsilon U_j^{-1}A, K_j) \,.
    \]
    By the Taylor expansion of $W$ at $U_j$ and by \Cref{lemma:simmetry_second_derivative} we get
    \begin{align}
        \varepsilon^2 Q_j(U_j^{-1} \sym(A)) + o(\varepsilon^2) & \geq C \dist^2(U_j + \varepsilon U_j^{-1}A, K_j) \\
        & \geq C \dist^2(\Id + \varepsilon U_j^{-1}AU_j^{-1}, \SO(3))\\
        & = C \varepsilon^2 |\sym(U_j^{-1}AU_j^{-1})|^2 + o(\varepsilon^2)\\
        & = C \varepsilon^2 |U_j^{-1}\sym(A)U_j^{-1}|^2 + o(\varepsilon^2)\\
        & \geq C \varepsilon^2 |\sym(A)|^2 + o(\varepsilon^2)\,.
    \end{align}
    Dividing by $\varepsilon^2$ and passing to the limit as $\varepsilon \to 0$ we conclude.
\end{proof}

It follows that the function $L_j: \r^{2 \times 2}_{\sym} \to \r^3$ that maps a matrix $D \in \r^{2 \times 2}_{\sym}$ into the vector $L_j(D)$ such that
\begin{equation}\label{eq:definition_L}
    \bar Q_j(D) = Q_j(U_j^{-1}(D + L_j(D) \otimes e_3 + e_3 \otimes L_j(D))) 
\end{equation}
is well-defined owing to the coercivity of $Q_j$.
Moreover, $ L_{j} $ is linear. 
Indeed, $ Q_{j}(U_{j}^{-1}(\cdot)) $ can be represented by a fourth-order tensor $ \mathbb{C} $ so that
\begin{equation}
    Q_{j}(U_{j}^{-1}(D + x \otimes e_{3} + e_{3} \otimes x)) = \mathbb{C}(D + x \otimes e_{3} + e_{3} \otimes x):(D + x \otimes e_{3} + e_{3} \otimes x) \,.
\end{equation}
In particular, $ L_{j}(D) $ is the unique vector that solves the linear optimality conditions
\begin{equation}
    \mathbb{C}(D + x \otimes e_{3} + e_{3} \otimes x):(e_{i} \otimes e_{3} + e_{3} \otimes e_{i}) = 0, \qquad i = 1, 2, 3.
\end{equation}

%% file: compactness.tex
\section{Compactness estimates}\label{section:compactness}

In order to study the $\Gamma$-convergence of the functionals $E_h^\alpha$ we first need to establish compactness for sequences of deformations that have bounded rescaled energy. It is clear that the elastic part of the rescaled energy forces the rescaled deformation gradient to approach in the limit the union of the wells $K$. However, we would like to prove that the rescaled gradients are actually getting closer to a single well $K_i$. This is precisely ensured by the penalty term. In the following result, inspired by \cite{ALICANDRO2018}, we give a precise meaning to this statement.

\begin{prop}\label{prop:compactness}
    Let $\alpha \geq 2$. Let $(y_h) \subset W^{1,2}(\Omega; \r^3)$ be a sequence such that
    \begin{equation}
        \lim_{h \to 0} h^\alpha E^\alpha_h(y_h) = 0 \label{eq:hypothesis_compactness}\,.
    \end{equation}
    Then for $h \ll 1$ there are an index $i_h \in \{1, \dots, l\}$ and two constants $\delta, C(\delta) > 0$ such that
    \begin{align}
        & \int_{\Omega_{h}^{i_{h}}} \dist^2(\nabla_h y_h, K_{i_h}) \, dx \leq C(\delta)h^\alpha E^\alpha_h(y_h) \label{eq:compactness_near_well} \,,\\
        & \int_{\Omega \backslash \Omega_h^{i_h}} \dist^2(\nabla_h y_h, K_{i_h}) \, dx \leq C(\delta)h^\alpha[{(E^\alpha_h(y_h))}^{\theta} + E^\alpha_h(y_h)] \label{eq:compactness_far_well}\,,
    \end{align}
    where $\Omega_h^{i_h} = \{x \in \Omega \, : \, \dist(\nabla_h y_h(x), K_{i_h}) \leq \delta \}$ and
    \begin{equation}\label{eq:theta}
        \theta = \begin{cases}
            \frac{3}{2} \quad & \text{if} \, \, q = 2 \,,\\
            \frac{5}{3} \quad & \text{otherwise} \,.
        \end{cases}
    \end{equation}
\end{prop}

The proof follows the same arguments used in \cite[Theorem 2.3]{ALICANDRO2018}. For the convenience of the reader, we give a sketch of the proof.

\begin{proof}
    Fix $\delta < 1 \land \min_{i \neq j}\{\frac{1}{2}\dist(K_i, K_j)\}$ and define
    \[
        \Omega_h = \{x \in \Omega \colon \dist(\nabla_h y_h(x), K) \leq \delta\} = \bigcup_{i = 1}^l \Omega_h^i \,,
    \]
    where
    \begin{equation}
        \Omega_h^i = \{x \in \Omega \colon \dist(\nabla_h y_h(x), K_i) \leq \delta\} \,.
    \end{equation}
    Note that, by \ref{item:lower_bound_W}, condition \eqref{eq:compactness_near_well} holds with an arbitrary $ j \in \{1, \dots, l\} $ in place of $ i_{h} $. 
    
    Define $\widetilde W(A) = f_q(\dist(A, K))$ for $A \in \r^{3 \times 3}$, where we recall that $f_q(t) = t^2 \land t^q$. Let
    \begin{equation}
        d_{\widetilde W}(A, B) = \inf\left\{\int_0^1(\widetilde W(\xi(s)))^{\frac{l}{m p}}|\xi'(s)| \, ds \colon \xi \in C^1([0, 1]; \r^{d \times d}) \text{ s.t. } \xi(0) = A, \xi(1) = B \right\} \,,
    \end{equation}
    where $l, m > 1$ are such that $\frac{1}{l} + \frac{1}{m} = 1$.
    By \eqref{eq:hypothesis_compactness} and \ref{item:lower_bound_W} we have that
    \begin{equation}\label{eq:far_estimate_measure}
        |\Omega \backslash \Omega_h| \leq h^{\alpha}E_{h}(y_{h})\to 0.
    \end{equation}
    Hence, there exists $i_h \in \{1, \dots, l\}$ such that $|\Omega_h^{i_h}| \geq C(\delta)$.
    We define 
    \begin{equation}
        g_h(x) = (d_{\widetilde W}(\nabla_h y_h(x), K_{i_h})  - \tilde \delta) \lor 0 \,,
    \end{equation}
    where $\tilde \delta = \delta^{2 \frac{l}{mp} + 1}$. One can check that $g_h \equiv 0$ on $\Omega_h^{i_h}$. Set $\beta = \frac{p}{l}$ and choose $l$ in such a way that $\beta < 3$.
    Observe that $ g_h \in W^{1, \beta }(\Omega) $ and
    \begin{equation}
        |\nabla g_{h}|\leq(\widetilde W(\nabla_{h}y_{h}))^{\frac{l}{mp}}|\nabla (\nabla_{h}y_{h})| \,.
    \end{equation}
    Hence, applying the Sobolev embedding, the Poincar\'e inequality given in \cite[Lemma 2.2]{ALICANDRO2018}, and the Young inequality with exponents $ \frac{1}{l} $ and $ \frac{1}{m} $, we deduce that
    \begin{equation}\label{eq:estimate_gh}
        \begin{aligned} 
            \int_\Omega |g_h|^{\beta^\ast} \, dx & \leq C\left( \int_{\Omega} |\nabla g_{h}|^{\beta} \, dx\right)^{\frac{\beta^{\ast}}{\beta}} \leq \left( \int_{\Omega}(\widetilde W(\nabla_{h}y_{h}))^{\frac{1}{m}}|\nabla (\nabla_{h}y_{h})|^{\beta} \, dx \right)^{\frac{\beta^{\ast}}{\beta}} \\
                                                 & \leq C \frac{h^{\alpha \frac{\beta^\ast}{\beta}}}{\eta(h)^{\beta^\ast}}E_h^\alpha(y_h)^{\frac{\beta^\ast}{\beta}} \,,
        \end{aligned}
    \end{equation}
    where $\beta^\ast = \frac{3\beta}{3-\beta}$ is the critical Sobolev exponent in dimension $ 3 $. 
    Note that here we used the crucial information that $ |\Omega_{h}^{i_{h}}| \geq C(\delta) $ to deduce that the constant in the Poincar\'e inequality is independent of $ h $.
    Let
    \begin{equation}
        \tilde \Omega_h^i = \{x \in \Omega \colon d_{\widetilde W}(\nabla_h y_h, K_i) \leq 2\tilde \delta\} \,.
    \end{equation}
    By \cite[Lemma 2.6]{ALICANDRO2018}, we can refine the choice of $\delta$ in such a way that $\tilde \Omega_h^{i_h} \backslash \Omega_h^{i_h} \subset \Omega \backslash \Omega_h \subset \Omega \backslash \Omega_h^{i_h}$.
    Moreover, by \cite[Lemmas 2.5 and 2.6]{ALICANDRO2018} we have that $ \dist(\nabla_{h}y_{h}, K_{i_{h}}) \leq Cg_{h} $ on $\Omega \backslash \tilde \Omega_{h}^{i_{h}} $ and $ \dist(\nabla_{h}y_{h}, K_{i_{h}}) \leq C $ on $ \tilde \Omega_{h}^{i_{h}} \backslash \Omega_{h}^{i_{h}} $.
    Thus, writing $\Omega \backslash \Omega_h^{i_h}$ as $(\Omega \backslash \tilde \Omega_h^{i_h}) \cup (\tilde \Omega_h^{i_h} \backslash \Omega_h^{i_h})$ we deduce from \eqref{eq:far_estimate_measure}--\eqref{eq:estimate_gh} that
    \begin{equation}\label{eq:final_estimate_admlp}
        \begin{aligned}
            \int_{\Omega \backslash \Omega_h^{i_h}} \dist^{\beta^\ast}(\nabla_h y_h, K_{i_h}) \, dx & \leq C\int_{\Omega \setminus \tilde \Omega_{h}^{i_{h}}} |g_{h}|^{\beta^{\ast}} \, dx + C|\tilde \Omega_{h}^{i_{h}} \setminus \Omega_{h}^{i_{h}}| \\
                                                                                                    & \leq C\left(\frac{h^{\alpha \frac{\beta^\ast}{\beta}}}{\eta(h)^{\beta^\ast}}E_h^\alpha(y_h)^{\frac{\beta^\ast}{\beta}} + h^\alpha E_h^\alpha(y_h) \right) \,.
        \end{aligned}
    \end{equation}
    
    If $q \neq 2$, we choose $l = \frac{5}{6}p$. Note that $l > 1$ by \ref{item:bound_p}. Then $\beta = \frac{6}{5}$, $\beta^\ast = 2$, and $\frac{\beta^\ast}{\beta} = \frac{5}{3} = \theta$. By \ref{item:lower_bound_penalties} and \eqref{eq:final_estimate_admlp} we get \eqref{eq:compactness_far_well}.

    If $q = 2$, we choose $l = p$, so that $\beta = 1$, $\beta^\ast = \frac{3}{2}$ and $\frac{\beta^\ast}{\beta} = \theta = \frac{3}{2}$. 
    Fix a constant $M > 0$ such that $K \subset B_M(0)$ and define $B_h^M = \{x \in \Omega \colon |\nabla_h y_h(x)| \leq M\}$. Writing $\Omega \backslash \Omega_h^{i_h}$ as $((\Omega \backslash \Omega_h^{i_h}) \cap B_h^M) \cup ((\Omega \backslash \Omega_h^{i_h}) \backslash B_h^M)$, by \ref{item:lower_bound_W} we deduce that
    \begin{equation}\label{eq:final_estimate_admlp_q2}
        \begin{aligned}
            \int_{\Omega \backslash \Omega_h^{i_h}} \dist^2(\nabla_h y_h, K_{i_h}) \, dx & \leq C\int_{(\Omega \backslash \Omega_h^{i_h}) \cap B_h^M} \dist^{\frac{3}{2}}(\nabla_h y_h, K_{i_h}) \, dx \\
            & \hphantom{\geq} \, \, + C \int_{(\Omega \backslash \Omega_h^{i_h}) \backslash B_h^M} W(\nabla_h y_h) \, dx \,.
        \end{aligned}
    \end{equation}
    Then, \eqref{eq:compactness_far_well} follows by \ref{item:lower_bound_penalties} and \eqref{eq:final_estimate_admlp}--\eqref{eq:final_estimate_admlp_q2}.
\end{proof}

\begin{remark}
    At a first glance, it might seem that the hypothesis \eqref{eq:hypothesis_compactness} of \Cref{prop:compactness} does not depend on $\alpha$, while the thesis \eqref{eq:compactness_far_well} does. However, \ref{item:lower_bound_penalties} prescribes a dependence on $\alpha$ of the penalty term coefficient $\eta(h)$. This is particularly clear examining the sketch of the proof above. 
\end{remark}

\begin{cor}\label{cor:compactness}
    Let $\alpha \geq 2$. Let $(y_h) \subset W^{1,2}(\Omega; \r^3)$ be a sequence such that
    \begin{equation}
        \lim_{h \to 0} h^\alpha E^\alpha_h(y_h) = 0 \,.
    \end{equation}
    Then for $h \ll 1$ there are an index $i_h \in \{1, \dots, l\}$ and two constants $\delta, C(\delta) > 0$ such that
    \begin{equation}
        \int_{\Omega} \dist^2(\nabla_h y_h, K_{i_h}) \, dx \leq C(\delta)h^\alpha E^\alpha_h(y_h).
    \end{equation}
\end{cor}

The following is a variant of the well-known rigidity estimate by Friesecke, James, and M\"uller (see \cite{FRIESECKE2002, FRIESECKE2006}), where the well $\SO(3)$ is replaced by $K_j = \SO(3)U_j$.

\begin{prop}\label{prop:approximated_rotation}
    Let $(y_h) \subset W^{1,2}(\Omega; \r^3)$ and let $j \in \{1, \dots, l\}$. Define
    \[
        D_{h, j} = \| \dist(\nabla_h y_h, K_j) \|_{L^2(\Omega)}\,.
    \]
    There are two maps $R_h \in L^\infty(S; \SO(3))$ and $\tilde R_h \in W^{1,2}(S; \r^{3 \times 3}) \cap L^\infty(S; \r^{3 \times 3})$ such that
    \begin{enumerate}[start=1,label={(R\arabic*)}]
        \item $\|\nabla_h y_h - R_h U_j\|_{L^2(\Omega)} \leq C D_{h, j}$\label{item:approximated_rotation_convergence_scaled_gradient},
        \item $\|\nabla \tilde R_h\|_{L^2(S)} \leq C h^{-1}D_{h, j}$\label{item:approximated_rotation_convergence_gradient},
        \item $\|\tilde R_h - R_h\|_{L^2(S)} \leq C D_{h, j}$,
        \item $\|\tilde R_h - R_h\|_{L^\infty(S)} \leq C h^{-1}D_{h, j}$.
    \end{enumerate}
    Moreover, there exists a constant rotation $Q_h \in \SO(3)$ such that
    \begin{equation}
        \|R_h - Q_h\|_{L^2(S)} \leq Ch^{-1}D_{h, j} \,.\label{eq:approximated_constant_rotation}
    \end{equation}
    Finally, if $h^{-1}D_{h, j} \to 0$, then for $h \ll 1$ we can choose $\tilde R_h = R_h$.
\end{prop}

To prove this result it is enough to follow the same approach of \cite{FRIESECKE2006}. Indeed, the rigidity estimate \cite[Theorem 5]{FRIESECKE2006} holds also for a well of the form $\SO(3)U_j$ by a change of variable. Then, all the estimates in \cite[Theorem 6]{FRIESECKE2006} can be carried out in the same fashion.

\begin{remark}\label{remark:approximated_rotation_Lr}
	If $ r > 1 $, all the results of \Cref{prop:approximated_rotation} hold with the $L^2$ norm replaced by the $L^r$ norm and the factor $h^{-1}$ replaced by $h^{-\frac{2}{r}}$.
\end{remark}

\subsection{Compactness in the Kirchhoff case \texorpdfstring{$\bm{\alpha= 2}$}{\unichar{"03B1} = 2}}

In the case $\alpha = 2$ the $\Gamma$-limit is written in terms of the deformation gradient. In this section we show compactness for sequences of rescaled gradients and give a characterization of their limit. Firstly, we will need an explicit expression of the vector $\nu$ that appears in \eqref{eq:kirchoff_love}. 

\begin{lemma}\label{lemma:definition_nu}
    Let $U$ be a symmetric and positive definite matrix. Let $y \in W^{2,2}(S; \r^3)$ be such that $\nabla y^T \nabla y = (U^2)'$. Then there exists a unique function $\nu \in W^{1, 2}(S; \r^3)$ such that
    \[
        (\nabla y, \nu) U^{-1} \in \SO(3) \quad \text{a.e. in $S$} \,.
    \]
    In particular, $\nu$ is given by
    \[
        \nu = \frac{1}{|U^{-1}e_3|^2} \left[ \det(U^{-1}) (\partial_1 y \wedge \partial_2 y) - \sum_{k=1}^2(U^{-1}e_k \cdot U^{-1} e_3)\partial_k y \right] \,.
    \]    
\end{lemma}

\begin{proof}
    For the existence, it is enough to prove that $(\nabla y , \nu)^T (\nabla y, \nu) = U^2$ and $\det(\nabla y, \nu) > 0$. By the hypothesis on $y$ we need to prove that
    \begin{enumerate}[(i)]
        \item $\partial_1 y \cdot \nu = (U^2)_{13}$,
        \item $\partial_2 y \cdot \nu = (U^2)_{23}$,
        \item $\nu \cdot \nu = (U^2)_{33}$.
    \end{enumerate}
    For $j =1, 2$ we have
    \begin{align}
        \partial_j y \cdot \nu = -\dfrac{1}{|U^{-1}e_3|^2} \sum_{k=1}^3 (U^{-1}e_k \cdot U^{-1}e_3)(U^2)_{jk} + (U^2)_{j3} = (U^2)_{j3}\,.
     \end{align}
    To complete the proof we observe that
    \[
        |Ue_1 \wedge Ue_2|^2 = |Ue_1|^2|Ue_2|^2 - (U e_1 \cdot U e_2)^2 = |\partial_1 y \land \partial_2 y|^2
    \]
    and, since $U e_1 \land U e_2 = \cof(U) e_3$
    \[
        |Ue_1 \wedge Ue_2|^2 = |\deter(U) U^{-1}e_3|^2 = \deter^2(U) |U^{-1} e_3|^2 \,.
    \]
    We are now ready to conclude:
    \begin{align}
        \nu \cdot \nu & = \dfrac{1}{|U^{-1}e_3|^2} + \dfrac{(U^{-1}e_1 \cdot U^{-1}e_3)}{|U^{-1}e_3|^4}\sum_{k=1}^3 (U^{-1}e_k \cdot U^{-1}e_3)(U^2)_{1k} \\
        & \hphantom{=} + \dfrac{(U^{-1}e_2 \cdot U^{-1}e_3)}{|U^{-1}e_3|^4}\sum_{k=1}^3 (U^{-1}e_k \cdot U^{-1}e_3)(U^2)_{2k} \\
        & \hphantom{=} - \dfrac{1}{|U^{-1}e_3|^2} \sum_{k=1}^3 (U^{-1}e_k \cdot U^{-1}e_3) (U^2)_{3k}  + (U^2)_{33}\\
        & = \dfrac{1}{|U^{-1}e_3|^2} - \dfrac{1}{|U^{-1}e_3|^2} + (U^2)_{33} = (U^2)_{33} \,.
    \end{align}
    To show that $\det(\nabla y, \nu) > 0$, it is sufficient to note that
    \begin{equation}
        (\nabla y, \nu) = (\partial_1 y, \partial_2 y, \partial_1 y \land \partial_2 y) \cdot \begin{pmatrix}
            1 & 0 & -\frac{1}{|U^{-1}e_3|^2}(U^{-1}e_1 \cdot U^{-1}e_3)\\
            0 & 1 & -\frac{1}{|U^{-1}e_3|^2}(U^{-1}e_2 \cdot U^{-1}e_3)\\
            0 & 0 & \frac{\det(U^{-1})}{|U^{-1}e_3|^2}
        \end{pmatrix} \,,
    \end{equation}
    and that the determinant of both matrices in the right-hand side is positive.

    To prove uniqueness we observe that, for any choice of two different rotations $R_1, R_2 \in \SO(3)$, we have $\rank(R_1-R_2) \neq 1$. Indeed, given a vector $v \in \r^3$ we have 
    \[
    	(R_1 - R_2)v = 0 \iff R_1v = R_2 v \iff R_2^T R_1 v = v \,.
    \]
    Since $R_2^T R_1 \in \SO(3)$ and $R_2^TR_1 \neq \Id$, we deduce that $v \in \ker(R_1 - R_2)$ if and only if $v$ is parallel to the rotation axis, that is, $v$ belongs to a $1$-dimensional subspace. In particular, $\rank(R_1 - R_2) = 2$.
    Suppose now that there is another vector $\tilde \nu$ such that
    \[
    	(\nabla y, \tilde \nu)U^{-1} \in \SO(3) \, \, \text{a.e..}
    \]
    Then, 
    \[
    	 (0, \nu - \tilde \nu)U^{-1} = (\nu - \tilde \nu) \otimes U^{-1} e_3
    \]
    coincides a.e. with the difference of two rotations and has rank 1 whenever $\nu \neq \tilde \nu$. Thus, $\nu = \tilde \nu$ almost everywhere.
\end{proof}

We move now to the proof of the first part of \Cref{teo:gamma_convergence_alpha_2}.

\begin{proof}[Proof of \Cref{teo:gamma_convergence_alpha_2}--\ref{item:compactness_alpha_2}]
    By \Cref{cor:compactness} there is a sequence of indices $i_h \in \{1, \dots, l\}$ such that
    \[
        \|\dist(\nabla_h y_h, K_{i_h})\|_{L^2} \leq Ch \,.
    \]
    Upon a further subsequence, since $ i_{h} $ takes values in a finite set, we can suppose $i_h$ to be constant and equal to $j$. Construct the sequences $R_h$ and $\tilde R_h$ as in \Cref{prop:approximated_rotation}. Clearly, $\tilde R_h$ is bounded in $W^{1,2}(S; \r^{3 \times 3})$ thus it converges weakly, at least along a subsequence, to a map $R \in W^{1,2}(S; \r^{3 \times 3})$. Hence, we have $R_h \to R$ in $L^2(S; \SO(3))$, so $R$ takes values in the set of rotations. Consequently, $\nabla_h y_h U_j^{-1} \to R$ in $ L^{2}(S; \r^{3 \times 3}) $. By an application of the Poincaré-Wirtinger inequality we have 
    \[
    	\|y_h - c_h\|_{W^{1, 2}} \leq C\|\nabla_h y_h\|_{L^2} \leq C \,,
    \]
    where
    \[
    	c_h = \dfrac{1}{|\Omega|}\int_\Omega y_h(x) \, dx \,. 
    \]
    Thus, $y_h - c_h$ converges weakly (possibly along a subsequence) to some map $y \in W^{1, 2}(\Omega; \r^3)$. Since $\frac{1}{h} \partial_3 y_h$ is bounded in $L^2(\Omega; \r^3)$, we have $\partial_3 y = 0$ and $y \in W^{1, 2}(S; \r^3)$. Hence, $\nabla y = (RU_j)^{1, 2}$, $y \in W^{2, 2}(S; \r^3)$ and $\nabla_h' y_h \to \nabla y$. Lastly, since $\nu$ is uniquely determined by the condition $(\nabla y, \nu)U_j^{-1} \in \SO(3)$ a.e., the remaining part follows from \Cref{lemma:definition_nu}.
\end{proof}

\begin{cor}\label{cor:convergence_G_alpha_2}
    In the same setting of \Cref{teo:gamma_convergence_alpha_2}, there is a sequence $(R_h) \subset L^\infty(S; \SO(3))$ such that, up to a subsequence 
    \begin{equation}
        G_h := h^{-1}\left(R_h^T \nabla_h y_h - U_j\right) \todeb G \, \, \text{in} \, \, L^2(\Omega; \r^{3 \times 3}) \,. \label{eq:convergence_G_alpha_2}
    \end{equation}
    Moreover, $G^{1,2}$ is affine in $x_3$, that is
    \begin{equation}
        G^{1,2}(x', x_3) = G_0(x') + x_3 G_1(x') \,. \label{eq:linearity_G_alpha_2}
    \end{equation}
    Lastly,
    \begin{equation}
        (U_j G_1)' = \nabla y^T \nabla \nu \label{eq:G1_minor_gamma_equals_1} \,.
    \end{equation}
\end{cor}

\begin{proof}
    Arguing as in the proof of \Cref{teo:gamma_convergence_alpha_2}--\ref{item:compactness_alpha_2} we have 
    \[
        \|\dist(\nabla_h y_h, K_{j})\|_{L^2} \leq Ch \,,
    \]
    for some $j \in \{1, \dots, l\}$. Let $R_h \in L^\infty(S; \SO(3))$ be the map given by \Cref{prop:approximated_rotation}. Convergence \eqref{eq:convergence_G_alpha_2} follows from \ref{item:approximated_rotation_convergence_scaled_gradient}. Moreover, arguing again as in the proof of \Cref{teo:gamma_convergence_alpha_2}--\ref{item:compactness_alpha_2} we deduce that, up to subsequences, $R_h \to R \in L^2(S; \SO(3))$ and $\nabla_h y_h U_j^{-1} \to (\nabla y, \nu)U_j^{-1} = R$ in $L^2(\Omega; \r^{2 \times 2})$, where $\nu$ is given by \Cref{lemma:definition_nu}. Define
    \[
        H^s_h(x', x_3) = \dfrac{1}{s}(G_h(x',x_3+s) - G_h(x', x_3)) \,,
    \]
    for $s$ such that $x_3 + s \in I$. For $\alpha = 1, 2, 3$ and $\beta = 1, 2$ we have
    \begin{align}
        {(R_h(x') H_h^s(x', x_3))}_{\alpha \beta} &= \dfrac{1}{s}h^{-1}\left(\dpartial{y_{h, \alpha}}{x_\beta}(x', x_3 + s) - \dpartial{y_{h, \alpha}}{x_\beta}(x', x_3)\right) = \nonumber \\
        & = \dfrac{1}{s}\dpartial{}{x_\beta} \int_0^s \dfrac{1}{h} \dpartial{y_{h, \alpha}}{x_3}(x', x_3 + \sigma) \, d\sigma\,. \label{eq:integral_form_affinity}
    \end{align}
    The right-hand side converges strongly in $(W^{1,2}_0(\Omega))^\ast$ to $(\nabla \nu)_{\alpha \beta}$ as $h \to 0$.
    Indeed, one has that $ \|\partial_{i}g\|_{(W^{1,2}_{0})^{\ast}} \leq \|g\|_{L^{2}} $ for every $ g \in L^{2}(\Omega) $, where $ \|\cdot\|_{(W_{0}^{1, 2})^{\ast}} $ is the standard operatorial norm.
    The left-hand side converges weakly in $L^2(\Omega)$ to
    \begin{equation}
        {((\nabla y, \nu) U_j^{-1} H^s(x', x_3))}_{\alpha \beta} \,,
    \end{equation}
    where
    \begin{equation}
        {H^s(x', x_3)}_{\alpha \beta} = \dfrac{1}{s}({G(x', x_3+s)}_{\alpha \beta} - {G(x', x_3)}_{\alpha \beta}) \,.
    \end{equation}
    Since $L^2(\Omega)$ is continuously embedded in $(W^{1,2}_0(\Omega))^\ast$ we obtain 
    \[
        {H^s(x', x_3)}_{\alpha \beta} = {(U_j^{-1}(\nabla y, \nu)^T \nabla \nu)}_{\alpha \beta} \quad \forall \, \alpha =1, 2, 3 \quad \forall \, \beta = 1, 2 \,.
    \]
    In particular, the first two columns of $H^s$ are independent of $x_3$ and so the first two columns of $G$ are affine in $x_3$. Finally, we have
    \[
        (U_j G_1)' = ((\nabla y, \nu)^T \nabla \nu)_{1, 2} = \nabla y^T \nabla \nu \,,
    \]
    that proves \eqref{eq:G1_minor_gamma_equals_1}.
\end{proof}

\subsection{Compactness for the regime \texorpdfstring{$\bm{\alpha > 2}$}{\unichar{"03B1} > 2}}

In the case $\alpha > 2$ we will write the $\Gamma$-limit as a function of the in-plane and out-of-plane displacements. The next results give the correct scaling to extract their convergence. We start with a preliminary lemma.

\begin{lemma}\label{lemma:rotation_to_skew_matrix}
    Let $(F_h) \subset \r^{3 \times 3}$ be a sequence such that $|F_h - \Id| \leq Ch^\beta$ for some $\beta > 0$. Then there is a sequence $(P_h) \subset \SO(3)$ such that $\skw(P_h F_h) = 0$ and $|P_h - \Id| \leq Ch^\beta$.
\end{lemma}

\begin{proof}
For $h \ll 1$ the matrix $F_h$ is invertible with positive determinant and so its polar decomposition $F_h = R_h A_h$ is uniquely determined providing a matrix $P_h \in \SO(3)$ (i.e., $P_h = R_h^T$) such that $P_h F_h$ is symmetric. We have that 
\begin{equation}\label{eq:projection_polar_decomposition}
    |F_h - R_h| \leq |F_h - X| \quad \forall \, X \in \Ort(3) = \{B \in \r^{3 \times 3} \, : \, \det(B) = \pm 1 \} \,.
\end{equation}
Indeed, since $F_h = R_h A_h$, one has
\begin{equation}
    |F_h  - R_h|^2 = \Tr((F_h - R_h)^T(F_h - R_h)) = |F_h|^2 - 2\Tr(A_h) + 3\,,
\end{equation}
and similarly
\begin{equation}
    |F_h  - X|^2 = |F_h|^2 - 2\Tr(F_h^T X) + 3\,.
\end{equation}
Since $A_h$ is symmetric and positive definite we can write $A_h = O_h\Sigma_hO_h^T$ for some $O_h \in \Ort(3)$ and $\Sigma_h = \diag(\sigma_{h})$, with $\sigma_{h, i} > 0$. Thus, since the trace is invariant under circular shifts
\begin{align}
    |F_h  - R_h|^2 - |F_h  - X|^2 & = 2\Tr(F_h^T X - A_h) = 2\Tr(O_h \Sigma_h O_h^T R_h^T X - O_h \Sigma_h O_h^T) \\
    & = 2\Tr(\Sigma_h O_h^T R_h^T X O_h - \Sigma_h) = 2 \Tr(\Sigma_h Y_h - \Sigma_h) \\
    & = 2\sum_{i=1}^3 \sigma_{h, i}((Y_h)_{ii} - 1) \leq 0\,,
\end{align}
where $Y_h = O_h^T R_h^T X O_h \in \Ort(3)$.
Hence, by \eqref{eq:projection_polar_decomposition},
\[
    |P_h - \Id| = |R_h - \Id| \leq |R_h - F_h| + |F_h - \Id| \leq Ch^\beta \,.
\]
\end{proof}

\begin{prop}\label{prop:rescaling_convergence}
    Let $(y_h) \subset W^{1,2}(\Omega; \r^3)$ be a sequence of deformations such that
    \[
        \limsup_{h \to 0} E_h^\alpha(y_h) \leq C\,,
    \]
    where $\alpha > 2$. Then, for $h \ll 1$ there are an index $j \in \{1, \dots, l\}$, constant rotations $\bar R_h \in \SO(3)$, and constant vectors $c_h \in \r$ such that, setting $\tilde y_h$ as follows
    \begin{equation}
        \tilde y_h = \bar R_h y_h + c_h \,,
    \end{equation}
    there exist $\tilde R_h \in W^{1,2}(S, \SO(3))$ that satisfies:
    \begin{align}
        & \|\nabla_h \tilde y_h - \tilde R_h U_j\|_{L^2(\Omega)} \leq C h^{\gamma} \label{eq:convergence_rescaled_gradient} \,,\\
        & \|\tilde R_h - \Id\|_{L^2(S)} \leq Ch^{\gamma-1} \label{eq:convergence_approximated_rotation} \,,\\
        & \|\nabla \tilde R_h\|_{L^2(S)} \leq Ch^{\gamma-1} \label{eq:convergence_gradient_approximated_rotation} \,.
    \end{align}
    Moreover, there exists $A \in W^{1,2}(S, \r^{3 \times 3}_{\skw})$ such that the following convergences hold true, possibly along a non-relabeled subsequence:
    \begin{align}
        & A_h := h^{1-\gamma}(\tilde R_h - \Id) \todeb A && \text{in} \, \, W^{1,2}(S; \r^{3 \times 3}) \label{eq:convergence_Ah} \,,\\
        & h^{1-\gamma}(\nabla_h \tilde y_h - U_j) \to AU_j && \text{in} \, \, L^2(\Omega; \r^{3 \times 3}) \label{eq:convergence_rescaled_gradient_U} \,,\\
        & h^{2-2\gamma}\sym(\tilde R_h - \Id) \to \dfrac{A^2}{2} && \text{in} \, \, L^2(S; \r^{3 \times 3}) \label{eq:convergence_symmetric_part}\,.
    \end{align}
   	Lastly, convergences \ref{item:convergence_u}-\ref{item:convergence_v} of \Cref{teo:gamma_convergence}--\ref{item:compactness} hold true and $A$ has the following structure:
    \begin{equation}
        U_j A U_j = e_3 \otimes \nabla v - \nabla v \otimes e_3 \label{eq:characterization_A}\,.
    \end{equation}
\end{prop}

\begin{proof}
    Up to a subsequence extraction, by \Cref{cor:compactness} we have, for some index $j \in \{1, \dots, l\}$, that 
    \[
        D_{h, j} = \|\dist(\nabla_h y_h, K_j) \|_{L^2(\Omega)} \leq Ch^{\gamma} \,,
    \]
    where we recall that $\gamma = \frac{\alpha}{2}$. Let $Q_h$ and $R_h$ be, respectively, the constant rotation and the map whose existence is guaranteed by \Cref{prop:approximated_rotation}. Since $h^{-1}D_{h, j} \to 0$ we can suppose that $R_h \in W^{1,2}(S; \SO(3))$. Let $\tilde y_h = Q_h^T y_h $. By definition $\tilde y_h$ and $\tilde R_h = Q_h^T R_h$ satisfy \eqref{eq:convergence_rescaled_gradient}, \eqref{eq:convergence_approximated_rotation} and \eqref{eq:convergence_gradient_approximated_rotation}. Let
    \begin{equation}
        F_h = \dfrac{1}{|\Omega|}\dint_{\Omega} \nabla_h \tilde y_h {U}^{-1}_j \, dx \,.
    \end{equation}
    From \eqref{eq:convergence_rescaled_gradient} and \eqref{eq:convergence_approximated_rotation} we can deduce
    \begin{align}
        |F_h - \Id| & = \left|F_h - \dfrac{1}{|\Omega|}\int_{\Omega} \Id \, dx \right| \leq \dfrac{1}{|\Omega|}\int_{\Omega} |\nabla_h \tilde y_h {U_j}^{-1} - \Id| \, dx \\
        & \leq Ch^{\gamma -1} \,.
    \end{align}
    Therefore, by \Cref{lemma:rotation_to_skew_matrix} there exists a rotation $P_h \in \SO(3)$ such that $|P_h - \Id| \leq Ch^{\gamma -1}$ and
    \begin{equation}
        \int_{\Omega} \skw\left(P_h \nabla_h \tilde y_h  {U}_j^{-1}\right) \, dx = 0 \,.
    \end{equation}
    Therefore, redefining $\tilde y_h = P_h Q_h^T y_h$ and $\tilde R_h = P_h Q_h^T R_h$ we can additionally suppose that:
    \begin{equation}
        \int_{\Omega} \skw\left(\nabla_h \tilde y_h {U}_j^{-1}\right) \, dx = 0 \,. \label{eq:korn_condition}
    \end{equation}
    Moreover, we can choose the additive constant vector $c_h$ so that
    \begin{equation}
        \int_\Omega \left(\tilde y_h - U_j \begin{pmatrix}
            x' \\ h x_3
        \end{pmatrix}
        \right) \, dx = 0 \,. \label{eq:poincare_wirtinger_condition}
    \end{equation}
    From \eqref{eq:convergence_approximated_rotation} we deduce that $\|A_h\|_{L^2} \leq C$. Moreover, since $\nabla A_h = h^{1-\gamma} \nabla \tilde R_h$, by \eqref{eq:convergence_gradient_approximated_rotation} we have that $A_h$ is bounded in $W^{1,2}(S; \r^{3 \times 3})$. Hence, up to a subsequence, there is $A \in W^{1,2}(S; \r^{3 \times 3})$ such that \eqref{eq:convergence_Ah} holds true. By \eqref{eq:convergence_rescaled_gradient} and \eqref{eq:convergence_Ah} we get \eqref{eq:convergence_rescaled_gradient_U}. Using the well-known identity
    \begin{equation}
        {(Q - \Id)}^T(Q - \Id) = -2\sym(Q - \Id) \quad \forall \, Q \in \SO(3) \,,
    \end{equation}
    we obtain
    \begin{equation}
        \sym(A_h) = \dfrac{1}{2}h^{1-\gamma}{(\tilde R_h - \Id)}^T(\tilde R_h - \Id) \to 0 = \sym(A) \, \, \text{in} \, \, L^2(S; \r^{3 \times 3}) \,,
    \end{equation}
    that is, $A = - A^T$. In particular, we have
    \begin{equation}
        h^{2-2\gamma} \sym(\tilde R_h - \Id) = - \dfrac{1}{2}A_h^T A_h \to \dfrac{A^2}{2} \, \, \text{in} \, \, L^2(S; \r^{3 \times 3}) \,,
    \end{equation}
    that is precisely \eqref{eq:convergence_symmetric_part}. 
    
    To simplify the notation, we will write $U$ instead of $U_j$ for the rest of the proof. Consider $v_h$ as defined in \eqref{eq:definition_v} for the deformation $\tilde y_h$. By \eqref{eq:poincare_wirtinger_condition} and the Poincaré-Wirtinger inequality it follows that
    \begin{align}
        \|v_h\|_{W^{1,2}(S)}^2 & \leq C\|\nabla v_h\|_{L^2(S; \r^2)}^2 \leq Ch^{2-2\gamma}\int_{\Omega} |U\left(\nabla' \tilde y_h(x) - U^{1,2}\right)|^2 \, dx\\
        & \leq Ch^{2-2\gamma}\int_{\Omega} \left|\nabla_h' \tilde y_h - U^{1,2}\right|^2 \, dx \leq Ch^{2-2\gamma}\int_{\Omega} |\nabla_h \tilde y_h - U|^2 \, dx \leq C \,.
    \end{align}
    Hence, up to a subsequence, there is a map $v \in W^{1,2}(S)$ such that $v_h \todeb v$ in $W^{1,2}(S)$. Now, we want to show that $\nabla v_h \to \nabla v = e_3^T{(UAU)}^{1,2}$ from which we can deduce that $v \in W^{2,2}(S)$. Clearly,
    \[
        \nabla v_h = h^{1-\gamma} \int_{\Omega} e_3^TU\left(\nabla' \tilde y_h - U^{1,2}\right) \, dx \,.
    \]
    By \eqref{eq:convergence_rescaled_gradient_U}, it follows that
    \begin{equation}\label{eq:structure_gradient_v}
        \nabla v_h \to \nabla v = e_3^T{(UAU)}^{1,2} \,.
    \end{equation}
    Now we focus on the map $u_h$ defined as in \eqref{eq:definition_u}. We have
    \begin{align}
        \int_{\Omega} \skw(\nabla u_h) \, dx & = \min \left\{h^{-\gamma} \,, h^{2 - 2\gamma}\right\} \int_\Omega \skw\left(U_{1,2}(\nabla' \tilde y_h - U^{1,2})\right) \, dx\\
        & = \min \left\{h^{-\gamma} \,, h^{2 - 2\gamma}\right\} \int_\Omega U_{1,2}\skw\left(\nabla_h \tilde y_h U^{-1} - \Id\right) U^{1,2} \, dx\\
        & =\min \left\{h^{-\gamma} \,, h^{2 - 2\gamma}\right\} U_{1,2} \left[ \int_\Omega \skw\left(\nabla_h \tilde y_h U^{-1} - \Id \right) \, dx \right] U^{1,2}
    \end{align}
    and the last term is identically zero by \eqref{eq:korn_condition}. Therefore, we can apply Korn's inequality to deduce
    \begin{align}
        \|u_h\|_{W^{1,2}}^2 & \leq C \|\sym(\nabla u_h)\|_{L^2(S; \r^{2 \times 2})}^2 \\
        & \leq C \min \left\{h^{-2\gamma} \,, h^{4 - 4\gamma}\right\} \int_\Omega \left|\sym(\nabla_h \tilde y_h U^{-1}  - \Id)\right|^2 \, dx \\
        & \leq C \min \left\{h^{-2\gamma} \,, h^{4 - 4\gamma}\right\} \left[ \|\nabla_h \tilde y_h U^{-1} - \tilde R_h\|_{L^2}^2 + \|\sym(\tilde R_h - \Id)\|_{L^2}^2 \right] \\
        & \leq C \min \left\{h^{-2\gamma} \,, h^{4 - 4\gamma}\right\} \max\left\{h^{2\gamma} \,, h^{4\gamma-4}\right\} \leq C \,,
    \end{align}
    where we have used \eqref{eq:convergence_rescaled_gradient} and \eqref{eq:convergence_symmetric_part}. This proves the weak convergence of $u_h$, up to subsequences. Reasoning as before it is easy to note that
    \begin{equation}
        h^{1-\gamma} \max \left\{h^{\gamma} \,, h^{2\gamma-2}\right\} \nabla u_h \to U_{1,2} A U^{1,2} = {(U A U)}' \,.
    \end{equation}
    By the assumption $\gamma = \frac{\alpha}{2} > 1$ we have that $h^{1-\gamma} \max \left\{h^{\gamma} \,, h^{2\gamma-2}\right\} \to 0$, so $(U A U)' = 0$. Since $A$ is skew-symmetric, so is $U A U$ and this shows \eqref{eq:characterization_A} by \eqref{eq:structure_gradient_v}.
\end{proof}

\begin{remark}\label{remark:basis_u_v}
    In our setting $u_h$ and $v_h$ are the (suitably rescaled) components of the displacement $w_h$ with respect to the reference configuration $U_j \Omega$ in the basis $\{U_j^{-1} e_i \colon i = 1, 2, 3\}$, that is,
    \begin{equation}
        w_h = U_j^{-1}\begin{pmatrix}
            \max\{h^\gamma, h^{2 \gamma - 2}\} u_h \\
            h^{\gamma - 1} v_h
        \end{pmatrix} \,.
    \end{equation}
    Note that a basis of the tangent space to the midplane $U_j(S \times \{0\})$ is given by $\{U_j e_1, U_je_2\}$, while the normal direction is $U_j^{-1}e_3$. Thus, it may look more natural to define the in-plane and the out-of-plane displacement in terms of this basis, that is, to consider $\tilde u_h$ and $\tilde v_h$ such that
    \begin{equation}
        w_h = U_j \begin{pmatrix}
            \tilde u_h \\ 0
        \end{pmatrix} + \tilde v_h U_j^{-1} e_3 \,.
    \end{equation} 
    It is easy to see that
    \begin{align}
        u_h & = \min\{h^{-\gamma}, h^{2 - 2 \gamma}\} (U_j^2)' \tilde u_h \,,\\
        v_h & = h^{1 - \gamma} \left( \tilde v_h + (U_j^2)^{1, 2} \begin{pmatrix}
            \tilde u_h \\ 0
        \end{pmatrix} \cdot e_3\right) \,,
    \end{align}
    so that $h^{1-\gamma} \tilde v_h$ has the same limit as $v_h$, while $\min\{h^{-\gamma}, h^{2-2\gamma}\} \tilde u_h$ converges to some $\tilde u$, representing the same displacement as $u$ expressed in a different basis.
    Note that the same argument would apply defining $ \tilde u_{h} $ in terms of a basis of the form $ \{v_{1}, v_{2}, U_{j}^{-1}e_{3}\} $, with $ v_{1}, v_{2} \in \vspan \{U_{j}e_{1}, U_{j}e_{2}\}$.
\end{remark}

\begin{cor}\label{cor:convergence_G}
    In the same notation and hypothesis of \Cref{prop:rescaling_convergence}, there exists a map $G \in L^2(\Omega, \r^{3 \times 3})$ such that, up to a subsequence,
    \begin{equation}
        G_h := h^{-\gamma}\left(\tilde R_h^T \nabla_h \tilde y_h - U_j\right) \todeb G \, \, \text{in} \, \, L^2(\Omega; \r^{3 \times 3}) \,. \label{eq:convergence_G}
    \end{equation}
    Moreover, $G^{1,2}$ is affine in $x_3$, that is
    \begin{equation}
        G^{1,2}(x', x_3) = G_0(x') + x_3 G_1(x') \,. \label{eq:linearity_G}
    \end{equation}
    Finally,    
    \begin{equation}
        (U_j G_1)' = -\nabla^2 v\label{eq:G1_minor_gamma_greater_1} 
    \end{equation}
    and
    \begin{align}
        & (\sym(U_j G_0))' = \sym(\nabla u) && \text{if} \, \, \alpha > 4\,, \label{eq:G0_minor_b_greater_2} \\
        & (\sym(U_j G_0))' = \sym(\nabla u) + \dfrac{1}{2}|U_j^{-1} e_3|^2 \nabla v \otimes \nabla v && \text{if} \, \, \alpha  = 4\,, \label{eq:G0_minor_b_equals_2}\\
        & \nabla u + \nabla u^T + |U_j^{-1} e_3|^2 \nabla v \otimes \nabla v = 0 && \text{if} \, \, 2 < \alpha  < 4\,. \label{eq:G0_minor_b_less_2}
    \end{align}
\end{cor}

\begin{proof}
    We will write $U$ in place of $U_j$ to simplify the notation. Convergence \eqref{eq:convergence_G} follows immediately from \eqref{eq:convergence_rescaled_gradient}. To show that $G^{1,2}$ is affine, we study the difference quotient
    \[
        H^s_h(x', x_3) = \dfrac{1}{s}(G_h(x',x_3+s) - G_h(x', x_3)) \,,
    \]
    for $s$ such that $x_3 + s \in I$. Repeating the same computation as in \Cref{cor:convergence_G_alpha_2}, we deduce that for $\alpha = 1, 2, 3$ and $\beta = 1, 2$ we have
    \begin{equation}\label{eq:difference_quotient_Rh}
        {(\tilde R_h(x') H_h^s(x', x_3))}_{\alpha \beta} = \dfrac{1}{s} \dpartial{}{x_\beta} \int_0^s  h^{1-\gamma} \left(\dfrac{1}{h} \dpartial{\tilde y_{h, \alpha}}{x_3}(x', x_3 + \sigma) - E_{\alpha 3}\right)\, d\sigma \,.
    \end{equation}
    By \eqref{eq:convergence_rescaled_gradient_U}, the integral on the right hand-side converges strongly as $h \to 0$ in $L^2(\Omega)$ to
    \[
        \int_0^s {(AU)}_{\alpha 3} \, d\sigma = s{(AU)}_{\alpha 3} \,.
    \]
    Hence, the right-hand side of \eqref{eq:difference_quotient_Rh} converges strongly in $(W^{1,2}_0(\Omega))^\ast$ to $\dpartial{}{x_\beta} [{(AU)}_{\alpha 3}]$. By \eqref{eq:convergence_G}, the left-hand side of \eqref{eq:difference_quotient_Rh} converges weakly in $L^2(\Omega)$ to
    \begin{equation}\label{eq:difference_quotient_limit}
        {H^s(x', x_3)}_{\alpha \beta} := \dfrac{1}{s}({G(x', x_3+s)}_{\alpha \beta} - {G(x', x_3)}_{\alpha \beta}) \,. \label{eq:definition_H_s}
    \end{equation}
    Since $L^2(\Omega)$ is continuously embedded in $(W^{1,2}_0(\Omega))^\ast$, we obtain 
    \begin{equation} \label{eq:difference_quotient_structure}
        {H^s(x', x_3)}_{\alpha \beta} = \dpartial{}{x_\beta} [{(AU)}_{\alpha 3}] \quad \forall \, \alpha =1, 2, 3 \quad \forall \, \beta = 1, 2 \,.
    \end{equation}
    In particular, the first two columns of $H^s$ are independent of $x_3$ (recall that $A$ depends only on $x'$) and so the first two columns of $G$ are affine in $x_3$. 

    For the final part of the statement note that $G_h$ can be rewritten as follows
    \[
        G_h = \dfrac{\nabla_h \tilde y_h - U_j}{h^{\gamma}} - \dfrac{R_h U_j - U_j}{h^{\gamma}} + {(R_h - \Id)}^T\dfrac{\nabla_h \tilde y_h - R_h U_j}{h^{\gamma}} \,.
    \]
    Hence,  
    \begin{equation}\label{eq:decomposition_G}
        \begin{aligned}
            (U_j G_h)' & = {(U_j)}_{1,2} \dfrac{\nabla_h \tilde y_h U_j^{-1} - \Id}{h^{\gamma}} {(U_j)}^{1,2} - h^{\gamma - 2}{(U_j)}_{1,2} \dfrac{R_h - \Id}{h^{2\gamma - 2}} {(U_j)}^{1,2} \\
            & \hphantom{=} \, \, + \left[U_j{(R_h - \Id)}^T\dfrac{\nabla_h \tilde y_h - R_h U_j}{h^{\gamma}}\right]' \,.
        \end{aligned}
    \end{equation}
    If $\gamma > 2$, then $\min \{h^{-\gamma}, h^{2-2\gamma}\} = h^{-\gamma}$. Integrating with respect to $x_3$, taking the symmetric part and passing to the limit as $h \to 0$ we get
    \[
        \sym(U_j G_0)' = \sym(\nabla u) \,,
    \]
    where we have used \eqref{eq:convergence_rescaled_gradient}, \eqref{eq:convergence_Ah} and \eqref{eq:convergence_symmetric_part}. If $\gamma = 2$, passing to the limit in \eqref{eq:decomposition_G} we also get the term $-\frac{(UA^2U)'}{2}$. By the characterization \eqref{eq:characterization_A} of $UAU$ we get with some computation
    \begin{align}
        (UA^2U)' & = (UAU (U^{-1})^2 UAU)' = (UAU)_{1, 2} (U^{-1})^2 (UAU)^{1, 2} \\
        & = -\nabla v^T (U^{-1} e_3)^T U^{-1}e_3 \nabla v = - |U^{-1}e_3|^2 \nabla v \otimes \nabla v \,,
    \end{align}
    proving \eqref{eq:G0_minor_b_equals_2}. Lastly, if $1 < \gamma < 2$, we multiply both sides of \eqref{eq:decomposition_G} by $h^{2-\gamma}$ before passing to the limit. The left-hand side converges to $0$, while the right-hand side converges again to
    \[
        \sym(\nabla u) + \dfrac{1}{2}|U_j^{-1} e_3|^2 \nabla v \otimes \nabla v \,.
    \]
    Finally, note that $G_1 = H^1$ by \eqref{eq:difference_quotient_limit}. Thus, by \eqref{eq:difference_quotient_structure} for $\alpha, \beta = 1, 2$
    \begin{equation}
        {(U G_1)}_{\alpha \beta} = \sum_{k=1}^3 {(U)}_{\alpha k} \sum_{l=1}^3 \dpartial{A_{k l}}{x_\beta} {(U)}_{l3} = \dpartial{}{x_\beta} {(UAU)}_{\alpha 3} = - \dfrac{\partial^2 v}{\partial x_\beta \partial x_\alpha} \,.
    \end{equation}
\end{proof}

\begin{proof}[Proof of \Cref{teo:gamma_convergence}--\ref{item:compactness}]
    It immediately follows from \Cref{prop:rescaling_convergence} and \Cref{cor:convergence_G}.
\end{proof}

%% file: gamma-convergence.tex
\section{Proof of the \texorpdfstring{$\bm{\Gamma}$}{\unichar{"0393}}-Convergence results}\label{section:gamma-convergence}

We are now ready to complete the proofs of \Cref{teo:gamma_convergence_alpha_2} and \Cref{teo:gamma_convergence}. By the results of the previous section, we just need to prove the $\liminf$ inequality and the existence of recovery sequences.

\subsection{The \texorpdfstring{$\bm{\liminf}$}{liminf} inequality}

\begin{proof}[Proof of \Cref{teo:gamma_convergence_alpha_2}--\ref{item:liminf_inequality_alpha_2} and \Cref{teo:gamma_convergence}--\ref{item:liminf_inequality_alpha_2}]
    Define the\textsf{} matrix $G_h$ as in \Cref{cor:convergence_G} or \Cref{cor:convergence_G_alpha_2}, depending on the value of $\alpha$. Let $\tilde \Omega_h := \{x  \in \Omega \, : \, |G_h(x)| < h^{-1}\}$ and let $\chi_h$ be its characteristic function. Clearly, $\chi_h$ is bounded and $\chi_h \to 1$ in $L^1(\Omega)$. Thus, we have $\chi_h G_h \todeb G$ in $L^2(\Omega)$. Expanding $W$, we get
    \begin{align}
        W(\nabla_h y_h) \geq W(U_j + h^{\gamma} G_h) \geq \dfrac{1}{2} Q_j(h^{\gamma} G_h) - \omega(h^{\gamma} |G_h|) h^{\alpha} |G_h|^2 \,,    
    \end{align}
    where $\omega$ is the modulus of continuity of $D^2W$ at $U_j$. In particular
    \begin{align}
        E_h^\alpha(y_h) & \geq \dfrac{1}{2} \int_\Omega \left[Q_j(\chi_h G_h) - \omega(h^\gamma |\chi_h G_h|)|\chi_h G_h|^2 \right]\, dx\\
        & \geq \dfrac{1}{2} \int_\Omega Q_j(\chi_h G_h) \, dx - C \omega(h^{\gamma-1})\,.
    \end{align}
    Recall that $Q_j$ is weakly lower semicontinuous in $L^2(\Omega)$ by convexity. Thus, passing to the limit and applying \Cref{lemma:simmetry_second_derivative} we obtain
    \begin{align}
        \liminf_{h \to 0} E_h^\alpha(y_h) & \geq \dfrac{1}{2} \int_\Omega Q_j(G) \, dx = \int_\Omega Q_j(U_j^{-1}\sym(U_j G)) \, dx\\
        & \geq \int_\Omega \bar Q_j(\sym(U_j G)') \, dx \,.
    \end{align}
    Recall that
    \[
        \sym(U_j G(x', x_3))' = \sym(U_j G_0(x'))' + x_3 \sym(U_j G_1(x'))' \,.
    \]
    By \Cref{cor:convergence_G_alpha_2} and \Cref{cor:convergence_G} we conclude.
\end{proof}

\subsection{Recovery sequences}

We are left with the construction of the recovery sequences. Here, for a clearer exposition, we follow the reverse order and start with the recovery sequence for the case $\alpha > 4$. We will often use the big(small)-O notation, that has to be understood uniformly over $\Omega$.

\begin{proof}[Proof of \Cref{teo:gamma_convergence}--\ref{item:recovery_sequence_alpha_greater_equal_4}]
    Let $j \in \{1, \dots, l\}$. Suppose $\alpha > 4$. By a standard density argument it is sufficient to exhibit a recovery sequence for $u \in C^{\infty}(\bar S; \r^2)$ and $v \in C^\infty(\bar S)$. We define
    \[
        B_h(x', x_3) := \begin{pmatrix}
            h^\gamma u\\
            h^{\gamma - 1} v
        \end{pmatrix} - h^{\gamma} x_3 \begin{pmatrix}
            {\nabla v}^T\\
            0
        \end{pmatrix} + \dfrac{h^{\gamma + 1}}{2} x_3^2 \xi(x') + h^{\gamma + 1} x_3 \zeta(x') \,,
    \]
    and
    \[
        y_h(x', x_3) := U_j\begin{pmatrix}
            x'\\
            h x_3
        \end{pmatrix} + U_j^{-1}B_h(x', x_3) \,,
    \]
    where $\xi, \zeta$ are smooth functions independent of $x_3$ to be chosen. We immediately deduce that
    \begin{enumerate}[(i)]
        \item $h^{1-\gamma} v_h = v + \dfrac{1}{24}h^2\xi_3 \to v$ in $W^{1, 2}(\Omega)$,
        \item $ h^{-\gamma} u_h = u + \dfrac{1}{24}h \xi_{12} \to u$ in $W^{1, 2}(\Omega; \r^2)$.
    \end{enumerate}
    Computing the rescaled gradient, we get
    \[
        \nabla_h y_h = U_j + U_j^{-1} \nabla_h B_h(x', x_3) \,,
    \]
    where
    \[
        \nabla_h B_h = h^{\gamma}\left[ \begin{pmatrix}
            \nabla u - x_3\nabla^2 v & -h^{-1}{\nabla v}^T\\
            h^{-1} {\nabla v} & 0
        \end{pmatrix} + (x_3 \xi + \zeta) \otimes e_3 \right] + O(h^{\gamma+1}) \,.
    \]
    By construction, we have $\nabla_h y_h = M_hU_j$, where
    \[
        M_h =  \Id + U_j^{-1}\nabla_hB_hU_j^{-1} \,.
    \]
    We first compute $M^T_hM_h$ and obtain
    \begin{equation}
        M^T_hM_h = \Id + 2U_j^{-1} \sym(\nabla_h B_h) U_j^{-1}+ O(h^{2\gamma -2}) \,. \label{eq:product_Mh}
    \end{equation}
    Note that $\sym(\nabla_h B_h) = O(h^\gamma)$. Then, we develop $\sqrt{M^T_hM_h}$ near the identity to obtain
    \[
        \sqrt{M_h^TM_h} = \Id + U_j^{-1}\sym(\nabla_h B_h) U_j^{-1} + O(h^{2\gamma - 2}) \,.
    \]
    Let $\tilde W(P) = W(PU_j)$ for every $P \in \r^{3 \times 3}$. Clearly $\tilde W$ is frame indifferent and $D^2 \tilde W(\Id)[A]^2 = Q_j(AU_j)$. In particular $\tilde W(P) = \tilde W(\sqrt{P^TP})$. Note that $\gamma - 2 > 0$ whenever $\alpha > 4$. Developing $\tilde W$ near the identity, we obtain
    \begin{align}
        W(\nabla_h y_h) & = \tilde W(M_h) = \tilde W\left(\sqrt{M_h^TM_h}\right) = D^2\tilde W(\Id)[M_h-\Id]^2 + o(|F-\Id|^2)\\
        & = h^{2\gamma}\dfrac{1}{2}Q_j{\left[h^{-\gamma}U_j^{-1}\sym(\nabla_h B_h) + O(h^{\gamma - 2})\right]}^2+ o(h^{2\gamma}) \,.
    \end{align}
    Observing that $\nabla_h^2 y_h = U_j^{-1}\nabla_h^2 B_h = O(h^{\gamma-1})$, we get by assumption \ref{item:convergence_penalties}
    \[
        \lim_{h \to 0^+} E_h^\alpha(y_h) = \dfrac{1}{2}\int_\Omega Q_j(U_j^{-1}\sym(B)) \, dx \,,
    \]
    where
    \[
        B = U_j^{-1}\left(\begin{pmatrix}
            \nabla u - x_3\nabla^2 v & 0\\
            0 & 0
        \end{pmatrix} + x_3\xi \otimes e_3 + \zeta \otimes e_3)\right) \,.
    \]
    Choose $\xi = -2L_j(\nabla^2 v)$ and $\zeta = 2L_j(\sym(\nabla u))$ where $L_j$ is the linear operator defined in \eqref{eq:definition_L}. Thus,
    \begin{align}
        \dfrac{1}{2}\int_\Omega Q_j(U_j^{-1}\sym(B)) \, dx & = \dfrac{1}{24}\int_S Q_j(U_j^{-1}(\nabla^2v + \sym(\xi \otimes e_3))) \, dx' \\
        & \hphantom{=} \, \, + \dfrac{1}{2}\int_S Q_j(U_j^{-1}(\sym(\nabla u) + \sym(\zeta \otimes e_3))) \, dx' \\
        & = \dfrac{1}{24}\int_\Omega \bar Q_j(\nabla^2 v) \, dx' + \dfrac{1}{8}\int_\Omega \bar Q_j(\nabla u^T + \nabla u) \, dx'
    \end{align}
    as desired. The factor $\frac{1}{24}$ is due to the integration of $x_3^2$ over $I$, while the mixed term gives no contribution since $x_3$ has zero mean over $I$.
    
    Suppose now $\alpha = 4$. We use the same recovery sequence and the same notation as in the case $\alpha > 4$. The main difference is that terms of order $h^{2\gamma - 2}$ in \eqref{eq:product_Mh} cannot be neglected. By some simple computation we get
    \[
        M_h^TM_h = \Id + 2h^2U_j^{-1}(P_h + V)U_j^{-1} + O(h^3)\,,
    \]
    where
    \begin{align}
    P_h & = \begin{pmatrix}
    \nabla u - x_3 \nabla^2 v & - h^{-1} \nabla v\\
    h^{-1} \nabla v & 0
    \end{pmatrix} + (x_3 \xi + \zeta) \otimes e_3 + \dfrac{1}{2}|U_j^{-1} e_3|^2 \nabla v \otimes \nabla v \,,\\
    V & = - \left(U_j^{-1}e_3 \cdot U_j^{-1} \nabla v^T \right) \sym(\nabla v \otimes e_3) + \dfrac{1}{2}\left| U_j^{-1} \nabla v^T\right|^2 e_3 \otimes e_3 \,.   
    \end{align}
    Hence, developing the square root we obtain
    \begin{align}
    \sqrt{M_h^TM_h} = \Id + h^2U_j^{-1}\sym(P_h+ V)U_j^{-1} + O(h^3) \,.
    \end{align}
    Note that $\sym(P_h)$ is independent of $h$. To conclude it is then sufficient to choose $\xi = -2L_j(\nabla^2 v)$ and 
    \begin{align}
    \zeta & = -\frac{1}{2}\left|U_j^{-1} \nabla v^T\right|^2 e_3 + \left( U_j^{-1}e_3 \cdot  U_j^{-1}\nabla v^T \right) \nabla v\\
    & \hphantom{=} \, \, + 2L_j\left(\sym(\nabla u) + \frac{1}{2}|U_j^{-1}e_3|^2\nabla v \otimes \nabla v\right) \,.   
    \end{align}
\end{proof}

In the case $2 < \alpha < 4$ the construction of the recovery sequence is different and involves the perturbation of an isometric embedding. In doing so, we will need some higher regularity of the boundary of $S$ to deal with the penalty term.

\begin{proof}[Proof of \Cref{teo:gamma_convergence}--\ref{item:recovery_sequence_2_alpha_4}]
    Firstly, suppose that $v \in C^{\infty}(\bar S)$ satisfies \eqref{eq:isometry_constraint}. By \Cref{remark:determinant_zero} we have $\det(\nabla^2 v) = 0$ in $S$. For $h \ll 1$ we can apply \Cref{lemma:existence_isometry} and construct a sequence $u_h \in W^{2, \infty}(S; \r^2)$ with $\|u_h\|_{W^{2, \infty}}$ uniformly bounded such that the map
    \[
        \tilde y_h (x') = U_j \begin{pmatrix}
            x' \\ 0
        \end{pmatrix} + h^{\gamma - 1}vU^{-1}_je_3 + h^{2\gamma - 2}U_j\begin{pmatrix}
            u_h \\ 0
        \end{pmatrix}
    \]
    satisfies $\nabla \tilde y^T_h \nabla \tilde y_h = (U^2_j)'$. Define $\nu_h$ as in \Cref{lemma:definition_nu} with $U = U_j$ for the map $\tilde y_h$. We consider the recovery sequence given by
    \[
        y_h(x', x_3) = \tilde y_h(x') + hx_3 \nu_h(x') +\dfrac{1}{2} h^{\gamma + 1} x_3^2 U_j^{-1} \xi(x') \,,
    \]
    where $\xi$ is a smooth function independent of $x_3$ to be determined. Observe that
    \[
    	v_h = v +h^{\gamma - 1} U_j^2 \begin{pmatrix}
    	u_h \\ 0
    	\end{pmatrix} \cdot e_3 + \dfrac{1}{24}h^2 \xi_3 \to v \quad \text{in} \, \, W^{1, 2}(S) \,.    
    \]
    We have
    \begin{equation}\label{eq:rescaled_gradient_recovery_sequence_24}
        \nabla_h y_h =(\nabla \tilde y_h, \nu_h) + h x_3(\nabla \nu_h, 0) + h^{\gamma} x_3 U_j^{-1} \xi \otimes e_3 + O(h^{\gamma + 1}) \,.
    \end{equation}
    Define $R_h = (\nabla \tilde y_h, \nu_h) U_j^{-1}$. By the definition of $\nu_h$, we have $R_h \in \SO(3)$. We rewrite $\nabla_h y_h$ as
    \[
        \nabla_h y_h = R_h(U_j + B_h) \,,
    \]
    where
    \[
        B_h := hx_3 U_j^{-1}\begin{pmatrix}
            \nabla \tilde y_h^T  \nabla \nu_h & 0 \\
            0 & 0
        \end{pmatrix} + h^{\gamma} x_3 R_h^T U_j^{-1} \xi \otimes e_3 + O(h^{\gamma + 1}) \,.
    \]
    Note that we used the fact that $\nu_h^T \nabla \nu_h = 0$, which follows from the constant norm of $\nu_h$. By \Cref{lemma:convergence_foundamental_form} with $\varepsilon = h^{\gamma - 1}$ we have
    \[
        \nabla \tilde y_h^T  \nabla \nu_h = -h^{\gamma - 1} \nabla^2 v + O(h^{2\gamma-2}) \,.
    \]
    Moreover, it is easy to check that $R_h = \Id + O(h^{\gamma-1})$. Thus, 
    \[
        B_h = h^\gamma U_j^{-1}\begin{pmatrix}
            -x_3 \nabla^2 v & 0 \\
            0 & 0
        \end{pmatrix} + h^{\gamma} x_3 U_j^{-1} \xi \otimes e_3 +O(h^{2\gamma - 1}) \,,
    \]
    and $\nabla_h y_h \to U_j$ in $L^2(\Omega; \r^{3 \times 3})$. In particular, $h^{-\gamma}B_h \to B$ in $L^2(\Omega; \r^{3 \times 3})$, where
    \[
        B = U_j^{-1}\begin{pmatrix}
            - x_3 \nabla^2 v & 0 \\
            0 & 0
        \end{pmatrix} + x_3 U_j^{-1} \xi \otimes e_3 \,.
    \]
    Developing $W$ we get
    \[
        W(\nabla_h y_h) = \dfrac{1}{2}Q_j(B_h) + o(h^{2\gamma}) \,.
    \]
    We are left to estimate the penalty term. 
    We have $\nabla_h^2 \tilde y_h = O(h^{\gamma - 1})$ and, by definition of $\nu_h$, $\nabla \nu_h = O(h^{\gamma - 1})$.
    By \eqref{eq:rescaled_gradient_recovery_sequence_24}, it follows that $\nabla^2_h y_h = O(h^{\gamma - 1})$.
    Hence, by \ref{item:convergence_penalties}, we deduce that
    \[
        \lim_{h \to 0^+} E_h^\alpha(y_h) = \dfrac{1}{2}\int_\Omega Q_j(B) \, dx \,.
    \]
    To conclude, it is sufficient to choose $\xi = -2L_j(\nabla^2 v)$, where $L_j$ is the linear operator defined in \eqref{eq:definition_L}. 
    
    For the general case of $v \in W^{2,2}(S)$ satisfying \eqref{eq:isometry_constraint} we apply \Cref{lemma:density_isometries} and a standard diagonal argument to conclude. 
\end{proof}

\begin{proof}[Proof of \Cref{teo:gamma_convergence_alpha_2}--\ref{item:recovery sequence_alpha_2}]
    Let $j \in \{1, \dots, l\}$. Firstly, suppose that $y \in C^\infty(\bar{S}; \r^3)$ satisfies $(\nabla y)^T \nabla y = (U_j^2)'$. Define
    \[
        y_h(x', x_3) = y(x') + hx_3 \nu(x') + \dfrac{1}{2} h^2 x_3^2 \xi(x') \,,
    \]
    where $\xi$ is a smooth function independent of $x_3$ to be chosen and $\nu$ is defined as in \Cref{lemma:definition_nu}. Let $R = (\nabla y, \nu) U_j^{-1}$. By construction $R \in \SO(3)$ a.e in $S$. Computing the rescaled gradient we get
    \begin{align}
        \nabla_h y_h & = (\nabla y, \nu) + hx_3( \nabla \nu, \xi) + \dfrac{1}{2}h^2 x_3^2(\nabla \xi, 0)\\
        & = R\left[U_j + h x_3 U^{-1}_j (\nabla y^T  \nabla \nu) + hx_3 R^T \xi \otimes e_3 + O(h^2) \right] \,.
    \end{align}
    It is clear that $\nabla_h y_h \to (\nabla y, \nu)$ in $L^2(S; \r^{3 \times 3})$. For $h \ll 1$ we expand $W$ and get
    \[
        W(\nabla_h y_h) = \dfrac{1}{2}h^2 x_3^2 Q_j\left(U_j^{-1}(\nabla y^T  \nabla \nu) + R^T \xi \otimes e_3\right) + o(h^2) \,.
    \]
    By the symmetry of $Q_j$ (see \Cref{lemma:simmetry_second_derivative}) we immediately deduce that 
    \[
        W(\nabla_h y_h) = \dfrac{1}{2}h^2 x_3^2 Q_j\left(U_j^{-1}(\nabla y^T \nabla \nu + \sym(U_jR^T \xi \otimes e_3))\right) + o(h^2) \,.
    \]
    Indeed, since $\partial_j(\nabla y^T  \nu) = 0$, the matrix $\nabla y^T \nabla \nu$ is symmetric by the following chain of equalities
    \begin{equation}\label{eq:chain_equalities_y_nu}
        \partial_i y \cdot \partial_j \nu = -\partial_{ij}y \cdot \nu = - \partial_{ji} y \cdot \nu = -\partial_j y \cdot \partial_i \nu \,.
    \end{equation}
    Set $\xi := RU_j^{-1}L_j(\nabla y^T  \nabla \nu)$, where $L_j$ is the linear operator defined in \eqref{eq:definition_L}. By construction, we have
    \[
        \dfrac{1}{h^2} \int_\Omega W(\nabla_h y_h) \, dx = \dfrac{1}{24} \int_S \bar Q_j(\nabla y^T  \nabla \nu) \, dx + o(1) \,.
    \]
    Clearly, the rescaled Hessian $\nabla_h^2 y_h$ is bounded in $L^\infty (\Omega; \r^{3 \times 3 \times 3})$. Hence, by \ref{item:convergence_penalties} we have
    \[
        \dfrac{\eta^p(h)}{h^2} \int_\Omega |\nabla_h^2 y_h|^p \, dx \to 0 \,,
    \]
    concluding the proof of the existence of a recovery sequence for a smooth $y$. 
    
    To conclude the proof, we first observe that $E_j^{\KL}$ is continuous with respect to the $W^{2, 2}$ topology.
    Let
    \[
        W^{2,2}_{\text{iso}, j}(S; \r^3) := \{y \in W^{2,2}(S; \r^3) \, : \ \nabla y^T \nabla y = (U^2_j)'\} \,.
    \]
    For every $y \in W^{2, 2}_{\text{iso}, j}(S; \r^3)$, arguing as in \eqref{eq:chain_equalities_y_nu}, we have $(\nabla y)^T \nabla \nu = - \nabla^2 y \nu$, where $\nu$ is defined as in \Cref{lemma:definition_nu}. Given a sequence $(y_n) \subset W^{2, 2}_{\text{iso}, j}(S; \r^3)$ such that $y_n \to y$ in $W^{2, 2}(S; \r^3)$ we have, up to a subsequence, $\nabla^2 y_n \to \nabla^2 y$ almost everywhere in $S$. Let $\nu_n$ and $\nu$ be defined as in \Cref{lemma:definition_nu} for $y_n$ and $y$, respectively. Then, $\nu_n \to \nu$ in $L^1(S; \r^3)$, thus, up to subsequences, $\nu_n \to \nu$ almost everywhere in $S$. Hence, $\nabla^2y_n \nu_n \to \nabla^2 y \nu$ and by dominated convergence $E_j^{\KL}(y_n) \to E_j^{\KL}(y)$. Moreover, the set of functions $y \in C^{\infty}(\bar S; \r^3)$ satisfying $(\nabla y)^T \nabla y = (U_j^2)'$ is dense in $W^{2, 2}_{\text{iso}, j}(S; \r^3)$. Indeed, let $\varepsilon > 0$ and pick an isometric embedding $y \in W^{2,2}_{\text{iso}, j}(S; \r^3)$. Define $\tilde y(x) = y(A^{-1}x)$, where $A := \sqrt{(U^2_j)'}$. Clearly, $\tilde y \in W^{2,2}(AS; \r^3)$ and
    \[
        \nabla \tilde y^T \nabla \tilde y = A^{-1} \nabla y^T \nabla y A^{-1} = \Id \,.
    \]
    The set $AS$ satisfies condition \eqref{eq:boundary_condition}. Hence, by \cite[Theorem 1]{HORNUNG2011} there exists a smooth isometric embedding $\tilde \phi$ for the flat metric $\Id$ such that $\|\tilde y - \tilde \phi\|_{W^{2,2}} \leq \varepsilon\sqrt{\det(A)}$. Defining $\phi(x) := \tilde \phi(Ax)$ we get $\phi \in C^{\infty}(\bar S; \r^3)$, $\nabla \phi^T \nabla \phi = A^2 = (U_j^2)'$, and $\|y - \phi\|_{W^{2,2}} \leq \varepsilon$. A standard diagonal argument allows one to conclude.
\end{proof}

\begin{remark}
    Observe that the argument used in \cite{FRIESECKE2002} to prove the existence of a recovery sequence cannot be applied here. Indeed, the truncation argument, which is the basis of the construction of \cite{FRIESECKE2002}, would lead to deformations with a low regularity, for which the penalty term cannot be written.
\end{remark}

%% file: forces.tex
\section{Convergence of minimizers with dead loads}\label{section:forces}

In this section we prove \Cref{teo:convergence_minimizers_alpha_2} and \Cref{teo:convergence_minimizers}. We start by showing that a sequence of deformations with bounded total energy has also bounded elastic energy.
For the convenience of the reader, we recall that
\[	
\mathcal{M}_{h} = \argmax_{K} F_{h} \,,
\]
where
\[
	F_{h}(A) = \int_S f_{h} \cdot A \begin{pmatrix}
	x' \\ 0
	\end{pmatrix} \, dx \,.
\]
Similarly, $ \mathcal{M} $ is defined as the set of maximizers over $ K $ of $ F $, where 
\[
	F(A) = \int_S f \cdot A \begin{pmatrix}
	x' \\ 0
	\end{pmatrix} \, dx \,.
\]

\begin{lemma}\label{lemma:compactness_forces}
	Let $\alpha \geq 2$ and $q > 1$. Suppose that $(y_h) \subset W^{1,2}(\Omega; \r^3)$ is a sequence of deformations that are quasi-minimizers for $J_h^\alpha$, that is,
	\[
	\lim_{h \to 0} \, (J_h^\alpha(y_h) - \inf J_h^\alpha) =0 \,.
	\]
	Then, $E_h^\alpha(y_h) \leq C$ for every $h > 0$.
\end{lemma}

\begin{proof}
	Firstly, we will prove that $h^\alpha E_h^\alpha(y_h) \to 0$. Fix $j \in \{1, \dots, l\}$ and let $Q_h$, $R_h$ be, respectively, the rotation and the map given by in \Cref{prop:approximated_rotation}. Define
	\[
	\tilde y_h = y_h  - Q_h U_j\begin{pmatrix}
	x' \\ h x_3
	\end{pmatrix} - c_h \,,
	\] 
	where
	\[
	c_h = \dfrac{1}{|\Omega|}\int_\Omega \left(y_h - Q_h U_j \begin{pmatrix}
	x' \\ hx_3
	\end{pmatrix}\right) \, dx \,.
	\]
	Using the test deformation
	\[
	(x', x_3) \mapsto U_j \begin{pmatrix}
	x' \\ hx_3
	\end{pmatrix}
	\]
	and the assumption \eqref{eq:convergence_forces}, we get $\inf_y J_h^\alpha(y) \leq Ch^{1-\gamma}$, so that
	\[
		J_h^\alpha(y_h) - C h^{1-\gamma} \leq J_h^\alpha(y_h) - \inf_y J_h^\alpha(y) \leq C \implies J_h^\alpha(y_h) \leq Ch^{1-\gamma} \,.
	\]
	We have
	\begin{equation}\label{eq:dim_forze_eq_1}
	\begin{aligned}
	h^\alpha E_h^\alpha(y_h) & = h^\alpha J_h^\alpha(y_h) + \int_\Omega f_h \cdot \tilde y_h \, dx  + \int_S f_h \cdot Q_hU_j \begin{pmatrix}
	x' \\ 0
	\end{pmatrix} \, dx \\
	& \leq Ch^{\gamma+1} + Ch^{\gamma+1}\|\nabla_h \tilde y_h\|_{L^q} \,.
	\end{aligned}
	\end{equation}
	Consider the set
	\[
	\tilde \Omega_h = \{x \in \Omega \colon \dist(\nabla_h y_h, K) \geq 1\}\,.
	\]
	By \Cref{prop:approximated_rotation} and \Cref{remark:approximated_rotation_Lr}, we deduce
	\begin{equation}\label{eq:dim_forze_eq_2}
		\begin{aligned}
		\|\nabla_h \tilde y_h\|_{L^q}^q & = \int_\Omega |\nabla_h y_h - Q_hU_j|^q \, dx \leq Ch^{-2} \int_\Omega \dist^q(\nabla_h y_h, K_j) \, dx\\
		& \leq Ch^{-2}\int_\Omega \dist^q(\nabla_h y_h, K) \, dx + Ch^{-2}\\
		& \leq Ch^{-2} \int_{\tilde \Omega_h} \dist^q(\nabla_h y_h, K) \, dx + Ch^{-2}\\
		& \leq Ch^{\alpha-2} E_h^\alpha(y_h) + Ch^{-2} \,,
		\end{aligned}
	\end{equation}
	where we have used \ref{item:lower_bound_W}. 
	Combining \eqref{eq:dim_forze_eq_1}-\eqref{eq:dim_forze_eq_2} we get
	\[
	h^\alpha E_h^\alpha(y_h) \leq Ch^{\gamma + 1} + Ch^{\gamma+1 - \frac{2}{q}}(h^\alpha E_h^\alpha(y_h))^{\frac{1}{q}}  +Ch^{\gamma + 1 - \frac{2}{q}} \,.
	\]
	Recalling that $q > 1$, by Young's inequality we deduce that
	\[
	h^\alpha E_h^\alpha(y_h) \leq Ch^{\gamma +1} + Ch^{q'(\gamma + 1 - \frac{2}{q})} + Ch^{\gamma + 1 - \frac{2}{q}} \to 0\,.
	\]
	Hence, we can apply \Cref{prop:compactness} and deduce that, at least along a subsequence, there is an index $j_0 \in \{1, \dots, l\}$ such that for $h \ll 1$
	\[
	\int_\Omega \dist^{2}(\nabla_h y_h, K_{j_0}) \,, dx \leq h^{\alpha}[E_h^{\alpha}(y_h) + (E_h^{\alpha}(y_h))^\theta] \,.
	\]

	We now show that $E_h^\alpha(y_h) \leq C$. To simplify the exposition, we suppose that $j_0 = j$. By \Cref{prop:approximated_rotation} we have
	\begin{equation}\label{eq:bound_yh_energy}
		\begin{aligned}
			\|\tilde y_h\|_{L^2}^2 & \leq C \|\nabla \tilde y_h\|_{L^2}^2 \leq Ch^{-2} \|\dist(\nabla_h y_h, K_j)\|^2_{L^2}\\
			& \leq Ch^{\alpha-2}[E_h^{\alpha}(y_h) + (E_h^{\alpha}(y_h))^\theta] \,.
		\end{aligned}
	\end{equation}
	Pick $\bar R_h U_{k_h} \in \mathcal{M}_h$ and define
	\[
	\tilde J^\alpha_h(y) := E^{\alpha}_h(y) - \dfrac{1}{h^\alpha} \int_\Omega f_h \cdot y \, dx + \dfrac{1}{h^\alpha} \int_S f_h \cdot \bar R_h U_{k_h} \begin{pmatrix}
	x' \\ 0
	\end{pmatrix} \, dx \,.
	\]
	Note that $y_h$ are quasi-minimizers for the functional $\tilde J_h^\alpha$. Moreover, using the test deformation
	\[
		(x', x_3) \mapsto \bar R_h U_{k_h}\begin{pmatrix}
		x' \\ hx_3
		\end{pmatrix} \,,
	\]
	we get $\inf_y \tilde J_h^\alpha(y) \leq 0$, so that
	\[
		\tilde J_h^\alpha(y_h) \leq \tilde J_h^{\alpha}(y_h) - \inf_y \tilde J_h^\alpha(y) \leq C \,.
	\]
	Hence,
	\begin{align}
	E_h^\alpha(y_h) & = \tilde J^\alpha_h(y_h) + \dfrac{1}{h^\alpha} \int_\Omega f_h \cdot y_h \, dx - \dfrac{1}{h^\alpha} \int_S f_h \cdot \bar R_h U_{k_h} \begin{pmatrix}
	x' \\ 0
	\end{pmatrix} \, dx' \\
	& = \tilde J^\alpha_h(y_h) + \dfrac{1}{h^\alpha} \int_\Omega f_h \cdot \tilde y_h \, dx + \dfrac{1}{h^\alpha} \int_S f_h \cdot (Q_h U_j  - \bar R_h U_{k_h}) \begin{pmatrix}
	x' \\ 0
	\end{pmatrix} \, dx'\\
	& \leq C + Ch^{-\gamma -1}\|f_h\|_{L^2}(E_h^{\alpha}(y_h) + (E_h^\alpha(y_h))^\theta )^{\frac{1}{2}}\\
	& \leq C + C(E_h^\alpha(y_h))^\frac{1}{2} + C(E_h^\alpha(y_h))^\frac{\theta}{2}\,,
	\end{align}
	where we have used \eqref{eq:bound_yh_energy} and the definition of $\mathcal{M}_h$.
	Since $0 < \frac{\theta}{2} < 1$, by an application of Young's inequality we conclude.
\end{proof}

We are in a position to conclude the proof of \Cref{teo:convergence_minimizers_alpha_2}.

\begin{proof}[Proof of \Cref{teo:convergence_minimizers_alpha_2}]
	By \Cref{lemma:compactness_forces} and \Cref{teo:gamma_convergence_alpha_2}--\ref{item:compactness_alpha_2} there exists $j \in \{1, \dots, l\}$ and $y \in W^{2, 2}(S; \r^3)$ satisfying $(\nabla y)^T \nabla y = (U_j^2)'$ such that, up to a subsequence, $(\nabla_h y_h) \to (\nabla y, \nu)$ in $L^2(\Omega; \r^{3 \times 3})$, where $\nu$ is defined as in \Cref{lemma:definition_nu}. Let
	\[
		c_h = \dfrac{1}{|\Omega|}\int_\Omega y_h \, dx
	\]
	and define $\tilde y_h := y_h - c_h$. By an application of the Poincaré-Wirtinger inequality, we deduce that $\tilde y_h \to y$ in $W^{1, 2}(S; \r^3)$. By the strong convergence of $\frac{1}{h^2}f_h$, we get
	\begin{equation}\label{eq:strong_convergence_force_term}
		\dfrac{1}{h^2} \int f_h \cdot y_h \, dx = \dfrac{1}{h^2} \int f_h \cdot \tilde y_h \, dx \to \int_S f \cdot y \, dx \,.
	\end{equation}
	By \Cref{teo:gamma_convergence_alpha_2}--\ref{item:liminf_inequality_alpha_2} we immediately deduce that
	\begin{equation}\label{eq:liminf_total_energy}
		\liminf_{h \to 0} J^2_h(y_h) \geq E^{\KL}_j(y) - \int_S f \cdot y \, dx = J_j^{\KL}(y)\,.
	\end{equation}
	Now let $i \in \{1, \dots, l\}$ and $y' \in W^{2, 2}(S; \r^3)$ be such that $(\nabla y')^T \nabla y' = (U_i^2)'$. Let $(\tilde y_h')$ be the recovery sequence for $y'$ provided by \Cref{teo:gamma_convergence_alpha_2}--\ref{item:recovery sequence_alpha_2}. Then 
	\begin{align}
		J_j^{\KL}(y) & \leq \liminf_{h \to 0} J_h^2(y_h) \leq \limsup_{h \to 0} J_h^2(y_h) = \limsup_{h \to 0} (\inf_y J_h^2(y))\\
		& \leq \limsup_{h \to 0} J_h^2(y_h') = \lim_{h \to 0} J_h^2(y_h') = J_i^{\KL}(y') \,,
	\end{align}
	concluding the proof.
\end{proof}

We move now to the case $\alpha > 2$. To start off, we recall some properties of the optimal rotations. We focus on the set $\mathcal{M}$ but all the arguments can be replicated for $\mathcal{M}_h$. Firstly, observe that given a rotation $R \in \SO(3)$ and a skew-symmetric matrix $W \in \r^{3 \times 3}_{\skw}$ the map $\phi(t) = Re^{tW}U_j$ takes values in $K$ for any choice of $j \in \{1, \dots, l\}$. In particular, if $RU_j \in \mathcal{M}$, then by differentiating the map $t \mapsto F(\phi(t))$ at $t = 0$ we obtain
\[
	F(RWU_j) = 0 \quad \text{and} \quad F(RW^2U_j) \leq 0 \quad \forall \, W \in \r^{3 \times 3}_{\skw}\,, \forall \, j \in \Lambda \,, \forall \, R \in \mathcal{R}^j \,.
\]
The missing information that we need to prove \Cref{teo:convergence_minimizers} is the well-posedness of the projection operator $P_h^j$ along a sequence of rotations converging to some element of $\mathcal{R}^j$. In order to do so we start proving some useful lemmas. The proof of the first one is trivial and is omitted.
\begin{lemma}\label{lemma:gamma_convergence_F}
	The sequence of functional $-\dfrac{1}{h^{\gamma+1}} F_h$ $\Gamma$-converges to $-F$. In particular, given a sequence $(R_h U_{k_h})$ such that $R_h U_{k_h} \in \mathcal{M}_h$ for every $h$ and $R_h U_{k_h} \to RU_j$ for some $j \in \{1, \dots, l\}$, we have $j \in \Lambda$ and $RU_j \in \mathcal{M}$.
\end{lemma}

\begin{lemma}\label{lemma:convergence_normals_R_h}
	Let $j \in \{1, \dots, l\}$. Let $(R_h) \subset \SO(3)$ such that $R_h \in \mathcal{R}^j_h$ for every $h$ and let $(W_h) \subset \r^{3 \times 3}_{\skw}$ be a sequence such that $R_hW_h$ is normal to $\mathcal{R}_h^j$ at the point $R_h$ and $|W_h| = 1$ for every $h$. Then, up to a subsequence, we have $R_h W_h \to RW$, where $R \in \mathcal{R}^j$, $W \in \r^{3 \times 3}_{\skw}$, $|W| = 1$, and $RW \in N \mathcal{R}^j_R$.
\end{lemma}

\begin{proof}
	Up to subsequences, we have that $R_h \to R$ and $W_h \to W$ with $R \in \SO(3)$, $W \in \r^{3 \times 3}_{\skw}$, and $|W| = 1$. By \Cref{lemma:gamma_convergence_F} we have that $R \in \mathcal{R}^j$, thus we just need to prove that $RW \in N\mathcal{R}^j_R$. Let $m = \dim \mathcal{R}^j$. By \ref{item:convergence_dimensions}, $m = \dim \mathcal{R}^j_h$ for $h \ll 1$. Consider an orthonormal basis of the tangent space $T \mathcal{R}_h^j{}_{R_h}$ given by $\{R_h W_h^1, \dots, R_h W_h^m\}$. Then, up to a subsequence, we have $R_h W_h^i \to R W^i$ for some $W^i \in \r^{3 \times 3}_{\skw}$. Clearly the matrices $\{R W^1, \dots, R W^m\}$ are orthonormal. Moreover, since
	\[
	0 = \dfrac{1}{h^{\gamma+1}}F_h(R_h (W^i_h)^2 U_j) \to F(R (W^i)^2 U_j) \quad \forall \, i =1, \dots, m 
	\]
	it follows that $\{R W^1 U_j, \dots, R W^m U_j\}$ is an orthonormal basis of $T\mathcal{R}^j_R$. 
	
	Consider now a matrix $M \in \r^{3 \times 3}_{\skw}$ such that $RM \in T\mathcal{R}^j_R$. Then, we can write $RM = \sum_{i = 1}^m \lambda_i RW^i$ for some $\lambda_1, \dots, \lambda_m \in \r$. Define $M_h = \sum_{i=1}^m \lambda_i W_h^i$. By construction, we have $R_h M_h \to RM$. Moreover, $R_h M_h$ is tangent to $\mathcal{R}_h^j$ at the point $R_h$, so that
	\[
		0 = R_hW_h : R_h M_h \to RW : RM \,.
	\]
	Since $M$ is arbitrary, this concludes the proof.
\end{proof}

\begin{remark}
	\Cref{lemma:convergence_normals_R_h} proves also that a sequence of tangent matrices to $\mathcal{R}_h^j$ at a point $R_h$ converges to a tangent matrix to $\mathcal{R}^j$ at the point $R$, where $R$ is the limit of $R_h$.
\end{remark}

\begin{prop}\label{prop:R_tilde_R}
	For $j \in \{1, \dots, l\}$ let
	\[
		\tilde{\mathcal{R}}^j = \{R \in \mathcal{R}^j \colon \exists \, (R_h) \subset \SO(3) \, \, \text{s.t.} \, \, R_h \in \mathcal{R}_h^j \, \, \text{for every $h > 0$ and} \, \, R_h \to R\} \,.
	\]
	Then $\tilde{\mathcal{R}}^j = \mathcal{R}^j$.
\end{prop}

\begin{proof}
	Take $R \in \tilde{\mathcal{R}}^j$. We show that $\tilde{\mathcal{R}}^j$ is the image of $T\mathcal{R}^j_R$ through the map
	\[
		T\mathcal{R}^j_R \to \SO(3)\,, \quad RW \mapsto Re^W\,,
	\]
	in a neighborhood of $R$. In particular, this proves that $\tilde{\mathcal{R}}^j$ is an embedded submanifold of $\mathcal{R}^j$ and that the tangent spaces coincide, concluding the proof. 
	
	There exists a sequence $(R_h) \subset \SO(3)$ such that $R_h \in \mathcal{R}_h^j$ for every $h$ and $R_h \to R$. For $h \ll 1$, take an orthonormal basis $\{R_hW^1_h, \dots, R_hW^m_h\}$ of $T \mathcal{R}_h^j{}_{R_h}$, where $m = \dim\mathcal{R}_h^j = \dim \mathcal{R}^j$. Then $R_h W_h^i \to RW^i$ and since
	\[
		0 = \dfrac{1}{h^{\gamma+1}}F_h(R(W_h^i)^2U_j) \to F(R(W^i)^2U_j)\,,
	\]
	$\{RW^1, \dots, RW^m\}$ is an orthonormal basis of $T\mathcal{R}^j_R$. Now pick $W \in T\mathcal{R}^j_R$. By the convergence of the base we can construct a sequence $(W_h)$ such that $W_h \to W$ and $R_hW_h \in T\mathcal{R}_h^j$ for every $h \ll 1$. By \cite[Lemma 4.4]{MAOR2021} we have $R_he^{W_h} \in \mathcal{R}_h^j$. Thus, passing to the limit, we get by \Cref{lemma:gamma_convergence_F} that $Re^W \in \mathcal{R}^j$, that is, by definition $Re^W \in \tilde{\mathcal{R}}^j$.
\end{proof}

The above results guarantee the well-posedness of the projection.

\begin{prop}\label{prop:projection_well_posed}
	Let $(R_h)	\subset \SO(3)$ be a sequence such that $R_h \to R$ and $R \in \mathcal{R}^j$ for some $j \in \{1, \dots, l\}$. Then $P_h^j(R_h)$ is well-defined for $h \ll 1$.
\end{prop}

\begin{proof}
	It is sufficient to prove that $\dist_{\SO(3)}(R_h, \mathcal{R}_h^j) \to 0$. By \Cref{prop:R_tilde_R}, there exists a sequence of rotations $\tilde R_h$ such that $\tilde R_h \in \mathcal{R}_h^j$ for every $h$ and $\tilde R_h \to R$. Since
	\[
	\dist_{\SO(3)}(Q', Q) = |Q' - Q| + O(|Q'-Q|^2)
	\]
	for every $Q, Q' \in \SO(3)$, we have
	\[
		\dist_{\SO(3)}(R_h, \mathcal{R}_h^j) \leq \dist_{SO(3)}(R_h, \tilde R_h) = |R_h - \tilde R_h| + O(|R_h-\tilde R_h|^2) \to 0 \,.
	\]
\end{proof}

We are now ready to conclude the proof of \Cref{teo:convergence_minimizers}.

\begin{proof}[Proof of \Cref{teo:convergence_minimizers}]
	By \Cref{lemma:compactness_forces}, we have $E_h^\alpha(y_h) \leq C$. 
	
	We prove that $\bar R^TU_j \in \mathcal{M}$. This will also show that $j \in \Lambda$. Indeed, let $RU_k \in K$ and define
	\[
		y_h'(x', x_3) = RU_k \begin{pmatrix}
		x' \\ hx_3
		\end{pmatrix} \,.
	\]
	We get
	\begin{equation}\label{eq:minimization_forces_eq_1}
		\begin{aligned}
			J_h^\alpha(y_h) - \inf_y J_h^\alpha(y) & \geq J_h^\alpha(y_h) - J_h^\alpha(y_h') \\ 
			& \geq -\dfrac{1}{h^\alpha}\int_\Omega f_h \cdot y_h \, dx + \dfrac{1}{h^\alpha}\int_\Omega f_h \cdot RU_k \begin{pmatrix}
				x' \\ 0
			\end{pmatrix}
			\, dx\\
			& = - \dfrac{1}{h^\alpha} \int_\Omega f_h \cdot \bar R_h^TU_j^{-1}\begin{pmatrix}
				\max\{h^\gamma, h^{2\gamma - 2}\}u_h \\ h^{\gamma - 1} v_h
			\end{pmatrix} \, dx\\
			& \hphantom{=} \, \, + \dfrac{1}{h^\alpha} \int_\Omega f_h \cdot (RU_k - \bar R_h^TU_j)\begin{pmatrix}
			x' \\ 0
			\end{pmatrix} \, dx \,.
		\end{aligned}
	\end{equation}
	Multiplying by $h^{\gamma - 1}$ and passing to the limit we deduce that
	\[
		\int_\Omega f \cdot (RU_k - \bar R^TU_j)\begin{pmatrix}
		x' \\ 0
		\end{pmatrix} \, dx \leq 0 \,,
	\]
	where we have used convergence \eqref{eq:convergence_forces}.
	Since $RU_k \in K$ is arbitrary, we have $\bar R^TU_j \in \mathcal{M}$. 
	
	By \Cref{prop:projection_well_posed}, the projection $P_h^j(R_h^T)$ is well-defined for $h \ll 1$. Let $W_h \in \r^{3 \times 3}_{\skw}$ with $|W_h| = 1$ be such that
	\[
		P^j_h(\bar R_h^T) W_h \in N\mathcal{R}_h^j{}_{P_h^j(\bar R_h^T)}
	\]
	 and $\bar R_h^T = P_h^j(\bar R_h^T)e^{d_h W_h}$, where $d_h = \dist_{\SO(3)}(\bar R_h^T, P_h^j(\bar R_h^T))$. Then, by \eqref{eq:minimization_forces_eq_1} with $RU_k = P_h^j(\bar R_h^T)U_j$ and the fact that $F_h(P_h^j(\bar R_h^T)W_hU_j) = 0$ we have, expanding the exponential map,
	\begin{equation}\label{eq:competitor_minimizing_sequences}
		\begin{aligned}
				J_h^\alpha(y_h) - \inf_y J_h^\alpha(y) & \geq J_h^\alpha(y_h) - J_h^\alpha\left(P_h^j(\bar R_h^T)U_j \begin{pmatrix}
			x' \\ hx_3
			\end{pmatrix}\right) \\
			& \geq -C -\dfrac{d_h^2}{h^\alpha} \int_\Omega f_h \cdot P_h^j(\bar R_h^T) W_h^2 U_j \, dx + O\left(\dfrac{d_h^3}{h^{\gamma-1}}\right) \,.
		\end{aligned}
	\end{equation}
	Clearly $P_h^j(\bar R_h^T) \to \bar R^T$. Moreover, up to subsequences, $W_h \to W$ and by \Cref{lemma:convergence_normals_R_h}, $\bar R^T W \in N \mathcal{R}^j_{\bar R^T}$, thus $F(\bar R^TW^2 U_j) < 0$. We deduce by \eqref{eq:competitor_minimizing_sequences} that $d_h^2 = O(h^{\gamma - 1})$. In particular, there exists $\beta \geq 0$ such that $h^{1-\gamma}d_h^2 \to \beta^2$ and so 
	\[
		h^{\frac{1}{2}(1-\gamma)}(\bar R_h^T - P^j_h(\bar R_h^T)) = \dfrac{d_h}{h^{\frac{1}{2}(\gamma-1)}}P_h^j(\bar R_h^T)W_h + O\left(\dfrac{d_h^2}{h^{\frac{1}{2}(\gamma-1)}}\right) \to \beta \bar R^T W \,.
	\]

	We are left to prove the minimality property. We show it for $2 < \alpha < 4$. The other cases can be proved analogously. Firstly, note that 
	\begin{equation}\label{eq:convergence_force_term}
		\dfrac{1}{h^\alpha}\int_\Omega f_h \cdot y_h \, dx - \dfrac{1}{h^\alpha}F_h(P_h^j(\bar R_h^T)U_j) \to \int_S f \cdot \bar R^T U_j^{-1} e_3 v \, dx - F(\beta^2\bar R^TW^2U_j) \,.
	\end{equation}
	Indeed, we have the equality
	\begin{equation}
		\begin{aligned}
			\dfrac{1}{h^\alpha}\int_\Omega f_h \cdot y_h \, dx - F(P_h^j(\bar R_h^T)U_j) & = \dfrac{1}{h^\alpha} \int_\Omega f_h \cdot \bar R_h^TU_j^{-1}\begin{pmatrix}
			h^{2\gamma - 2}u_h \\ h^{\gamma - 1} v_h
			\end{pmatrix} \, dx\\
			& \hphantom{=} \, \, - \dfrac{1}{h^\alpha} \int_\Omega f_h \cdot (P_h^j(\bar R_h^T) - \bar R_h^T)U_j\begin{pmatrix}
			x' \\ 0
			\end{pmatrix} \, dx \,.
		\end{aligned}
	\end{equation}
	The first term behaves as follows:
	\begin{align}
		\dfrac{1}{h^\alpha} \int_\Omega f_h \cdot \bar R_h^TU_j^{-1}\begin{pmatrix}
		h^{2\gamma - 2}u_h \\ h^{\gamma - 1} v_h
		\end{pmatrix} \, dx & = \int_\Omega \dfrac{1}{h^{\gamma+1}} f_h \cdot \bar R_h^TU_j^{-1}\begin{pmatrix}
		h^{\gamma - 1}u_h \\ v_h
		\end{pmatrix} \, dx \\
		& \to  \int_S f \cdot \bar R^T U_j^{-1} e_3 v \, dx \,.
	\end{align}
	On the other hand, since $F_h(P_h^j(\bar R_h^T)W_h U_j) = 0$, we have
	\begin{equation}
		\dfrac{1}{h^\alpha}F_h(P_h^j(\bar R_h^T)U_j) = \dfrac{d_h^2}{h^{\gamma - 1}}\int_\Omega \dfrac{1}{h^{\gamma + 1}}f_h \cdot P_h^j(\bar R_h^T) W_h^2 U_j \begin{pmatrix}
			x' \\ 0
		\end{pmatrix} \, dx + O\left(\dfrac{d_h^3}{h^{\gamma-1}}\right) \,,
	\end{equation}
	that converges to $F(\beta^2 \bar R^T W^2 U_j)$. Thus, by \Cref{teo:gamma_convergence}--\ref{item:liminf_inequality} we have
	\begin{equation}\label{eq:liminf_inequality_forces}
		\liminf_{h \to 0} (J_h^\alpha(y_h) + F(P_h^j(\bar R_h^T)U_j)) \geq J_j^{CVK}(v, \bar R^T, \beta W) \,.
	\end{equation}

	Take an admissible quadruplet $(i, v', R', W')$ and let $y_h'$ be the recovery sequence for $v'$ provided by \Cref{teo:gamma_convergence}--\ref{item:recovery_sequence_2_alpha_4}. By \Cref{prop:R_tilde_R}, we can construct a sequence of rotations $R'_h$ converging to $R'$ such that $R_h' \in \mathcal{R}_h^i$. Note that, by \ref{item:frame_indifference},
	\[
		E_h^\alpha(y_h') = E_h^\alpha(R'_hy_h') \,,
	\]
	so that $E_h^\alpha(R'_h y_h) \to E_i^{CVK}(v')$. Moreover,
	\begin{align}
		\dfrac{1}{h^\alpha} \int_\Omega f_h \cdot R'_h y_h' - \dfrac{1}{h^\alpha}F_h(R'_h U_i)&  = \dfrac{1}{h^\alpha} \int_S f_h \cdot R'_hU_i^{-1}\begin{pmatrix}
			h^{2\gamma - 2}u_h' \\ h^{\gamma - 1} v_h'
		\end{pmatrix} \, dx \\
		& \to \int_S f \cdot R' U_i^{-1} e_3 v' \, dx \,,
	\end{align}
	where $u_h', v_h'$ are defined as in \eqref{eq:definition_u}--\eqref{eq:definition_v} with $y_h'$ and $U_i$ in place of $y_h$ and $U_j$, respectively. By hypothesis \ref{item:lambda_h_equals_lambda}, we have that $R_h'U_i \in \mathcal{M}_h$ for $h \ll 1$. Thus,
	\[
		F_h(P_h^j(\bar R_h^T)U_j) = F_h(R'_h U_i) \,,
	\]
	and $F(R'(W')^2U_i) \leq 0$. Hence, by \eqref{eq:liminf_inequality_forces} we deduce that
	\begin{align}
		J_j^{CVK}(v, \bar R^T, \beta W) & \leq \liminf_{h \to 0} \left(J_h^\alpha(y_h) + \dfrac{1}{h^\alpha}F_h(P_h^j(\bar R_h^T)U_j)\right)\\
		& \leq \limsup_{h \to 0} \left(J_h^\alpha(y_h) + \dfrac{1}{h^\alpha} F_h(P_h^j(\bar R_h^T)U_j)\right) \\
		& = \limsup_{h \to 0} \left(\inf_y J_h^\alpha(y) + \dfrac{1}{h^\alpha} F_h(P_h^j(\bar R_h^T)U_j)\right)\\
		& \leq \limsup_{h \to 0} \left(J_h^\alpha(R'_h y_h') + \dfrac{1}{h^\alpha} F_h(R_h'U_i)\right) \\
		& = J_i^{CVK}(v', R', 0) \leq J_i^{CVK}(v, R', W') \,,
	\end{align}
	that gives the minimality property.
\end{proof}

%% file: isometries.tex
\section{Construction of isometries}\label{app:appendix}

In this appendix we collect a few results regarding the construction of isometries, mainly based on refinements of \cite[Section 8]{FRIESECKE2006}.

\begin{lemma} \label{lemma:convergence_foundamental_form}
    Let $U \in \r^{3 \times 3}$ be a symmetric and positive definite matrix. For $\varepsilon > 0$ let $y_\varepsilon \in W^{2,2}(S; \r^3)$ be a map of the form
    \[
        y_\varepsilon(x') = U \begin{pmatrix}
            x' \\ 0
        \end{pmatrix} + \varepsilon vU^{-1}e_3 + \varepsilon^2 U\begin{pmatrix}
            u_\varepsilon\\
            0
        \end{pmatrix} \,,
    \]
    such that $\nabla y^T \nabla y = (U^2)'$, where $v \in W^{2,\infty}(S)$ and $u_\varepsilon$ is a bounded sequence in $ W^{2, \infty}(S;\r^2) $. Then, if $\nu_\varepsilon$ is defined as in \Cref{lemma:definition_nu}, we have
    \[
        \nabla y_\varepsilon^T \nabla \nu_\varepsilon = - \varepsilon \nabla^2 v + O(\varepsilon^2) \,.
    \]
\end{lemma}

\begin{proof}
    First, we compute $\partial_1 y_\varepsilon \wedge \partial_2 y_\varepsilon$. We have
    \[
        \partial_1 y_\varepsilon \wedge \partial_2 y_\varepsilon = Ue_1 \wedge Ue_2 + \varepsilon \left[ Ue_1 \wedge U^{-1}\begin{pmatrix}
            0 \\
            \partial_2 v
        \end{pmatrix} + U^{-1}\begin{pmatrix}
            0 \\ \partial_1 v
        \end{pmatrix} \wedge U e_2 \right] + O(\varepsilon^2) \,.
    \]
    Given the identity $Ua \wedge Ub = \det(U) U^{-1} (a \wedge b)$, we easily deduce that
    \begin{align}
        \partial_1 y_\varepsilon \wedge \partial_2 y_\varepsilon & = Ue_1 \wedge Ue_2 \\
        & \hphantom{=} \, \,+ \varepsilon \det(U) U^{-1} \left[\partial_2 v e_1 \wedge (U^{-1})^2 e_3 + \partial_1 v (U^{-1})^2e_3 \wedge e_2 \right] + O(\varepsilon^2) \,.
    \end{align}
    Computing the cross products, we get
    \[
         \partial_1 y_\varepsilon \wedge \partial_2 y_\varepsilon = Ue_1 \wedge Ue_2 - \varepsilon \det(U) U^{-1} \begin{pmatrix}
            |U^{-1} e_3|^2 \partial_1 v\\
            |U^{-1} e_3|^2 \partial_2 v\\
            -\sum_{k=1}^2 (U^{-1} e_k \cdot U^{-1} e_3) \partial_k v
         \end{pmatrix} + O(\varepsilon^2) \,.
    \]
    Let us set $q(x') := -\frac{1}{|U^{-1}e_3|^2}\sum_{k=1}^2 (U^{-1} e_k \cdot U^{-1} e_3) \partial_k v$. We have 
    \[
        \nabla (\partial_1 y_\varepsilon \wedge \partial_2 y_\varepsilon) = -\varepsilon |U^{-1}e_3|^2 \det(U) U^{-1}\begin{pmatrix}
            \nabla^2 v \\ \nabla q
        \end{pmatrix} + O(\varepsilon^2) \,.
    \]
     Now observe that for every triplet of indices $i, j, k = 1, 2$ we have
    \[
        \partial_i y_\varepsilon \cdot \partial^2_{jk} y_\varepsilon = O(\varepsilon^2) \,.
    \]
    Moreover, 
    \[
        \partial_i y_\varepsilon^T U^{-1} = e_i^T + O(\varepsilon) \,.
    \]
    Combining the previous equations with the definition of $\nu_\varepsilon$, the thesis follows.
\end{proof}

\begin{lemma}\label{lemma:existence_isometry}
    Let $U \in \r^{3 \times 3}$ be a symmetric and positive definite matrix and define $A = \sqrt{(U^2)'}$. Suppose that $S$ is simply connected. Let $v \in W^{2,2}(S)$ be such that $|U^{-1}e_3|\|\nabla v A^{-1}\|_{L^\infty} < 1$. Then, there exists $\phi \in W^{2,2}(S; \r^2)$ such that the map
    \begin{equation}\label{eq:definition_isometry}
        y(x') = vU^{-1}e_3 + U\begin{pmatrix}
            \phi\\
            0
        \end{pmatrix}
    \end{equation}
    satisfies $\nabla y^T \nabla y = (U^2)'$ if and only if $\det(\nabla^2 v) = 0$. Moreover, if $v$ satisfies the condition $|U^{-1}e_3|\|\nabla v A^{-1}\|_{L^\infty} < \frac{1}{2}$, then $\phi$ can be chosen such that $u := U  (\phi - x', 0)^T$ satisfies the following estimates:
    \begin{align}
        \|\nabla^2 u\|_{L^2} & \leq C \|\nabla v\|_{L^\infty} \|\nabla^2 v\|_{L^2} \,, \label{eq:bound_norm_isometry_L2}\\
        \|u\|_{W^{2,2}} & \leq C \|\nabla v\|_{L^\infty} \|\nabla^2 v\|_{L^2} + C \|\nabla v\|^2_{L^2} \label{eq:bound_norm_isometry_W12}\,.
    \end{align}
    Finally, if $v \in W^{2, \infty}(S)$, then $u \in W^{2,\infty}(S; \r^2)$ and the following inequality holds:
    \begin{equation}
        \|u\|_{W^{2, \infty}} \leq C(\|\nabla^2 v\|_{L^\infty}\|\nabla v\|_{L^\infty} + \|\nabla v \|_{L^\infty}^2) \,. \label{eq:bound_norm_isometry_W3}
    \end{equation}
\end{lemma}

\begin{proof}
    Observe that $y$ of the form \eqref{eq:definition_isometry} satisfies $\nabla y^T \nabla y = A^2$ if and only if
    \[
        \nabla \phi^T A^2 \nabla \phi + |U^{-1}e_3|^2 \nabla v \otimes \nabla v = A^2 \,.
    \]
    Defining $\phi' := A \phi$ and $v' = |U^{-1}e_3| v$ the previous equation is equivalent to solve
    \[
        (\nabla \phi' A^{-1})^T (\nabla \phi' A^{-1}) +  (\nabla v' A^{-1})^T \otimes (\nabla v' A^{-1})^T = \Id_2
    \]
    for the unknown $\phi'$. Define now $\tilde v' \in W^{2,2}(AS; \r^3)$ given by $\tilde v'(Ax') = v'(x')$. Solving the above equation for $\phi'$ is equivalent to solve
    \[
        (\nabla \tilde \phi')^T \nabla \tilde \phi' +  \nabla \tilde v' \otimes \nabla \tilde v' = \Id_2 \,,
    \]
    for $\tilde \phi' \in W^{1,2}(AS; \r^2)$. Thus, we can conclude by applying \cite[Theorem 7]{FRIESECKE2006}. Note that the transformations $v \leadsto v' \leadsto \tilde v'$ preserve the property of having null determinant of the Hessian, namely
    \[
        \det( \nabla^2 v) = 0 \iff \det( \nabla^2 v') = 0 \iff \det( \nabla^2 \tilde v') = 0 \,.
    \]
    Moreover, they preserve the boundedness of the gradient in the $L^\infty$ norm. Finally, define $\tilde u' = \tilde \phi' - x'$. It is easy to show that 
    \begin{align}
        & \|\nabla u\|_{L^2} = C\|\nabla \tilde u'\|_{L^2} \,,\\
        & \|u\|_{W^{1,2}} = C\|\tilde u'\|_{W^{1,2}} \,,\\
        & \|\nabla v\|_{L^\infty} = C\|\nabla \tilde v'\|_{L^\infty} \,,\\
        & \|\nabla^2 v\|_{L^2} = C\|\tilde v'\|_{L^2} \,.
    \end{align}
Estimates \eqref{eq:bound_norm_isometry_L2}-\eqref{eq:bound_norm_isometry_W12} then follow from \cite[Equations (180)--(181)]{FRIESECKE2006}. We are left to prove \eqref{eq:bound_norm_isometry_W3}. Note that it is sufficient to do the proof for $\tilde u'$, so we can suppose $U = \Id$. Then, the result follows from \cite[Lemma 4.3]{TOLOTTI2024}.
\end{proof}

\begin{lemma}\label{lemma:density_isometries}
    Let $U \in \r^{3 \times 3}$ be a symmetric and positive definite matrix and let $A = \sqrt{(U^2)'}$. Suppose that $S$ is simply connected and satisfies \eqref{eq:boundary_condition}. Let $v \in W^{2,2}(S)$ be such that 
    \[
        |U^{-1}e_3|\|\nabla v A^{-1}\|_{L^\infty} < 1
    \]
    and $\det(\nabla^2 v) = 0$. Then there is a sequence $(v_n) \subset C^\infty(\bar S)$ such that $\det(\nabla^2 v_n) = 0$ for every $n \in \n$ and $v_n \to v$ in $W^{2,2}(S)$.
\end{lemma}

\begin{proof}
    Applying \Cref{lemma:existence_isometry}, we construct $\phi$ such that
    \[
        y(x') = vU^{-1}e_3 + U\begin{pmatrix}
            \phi\\
            0
        \end{pmatrix}
    \]
    satisfies $\nabla y^T \nabla y = (U^2)'$. Using the same notation as in the proof of \Cref{lemma:existence_isometry}, we note that $\phi$ is defined in such a way that $\tilde \phi' := A\phi(Ax)$ gives an isometric embedding of $AS \subset \r^2$ for the Euclidean metric, that is, $\nabla \tilde y^T \nabla \tilde y = \Id_2$, where $\tilde y = (\tilde \phi', \tilde v')^T$. Since \eqref{eq:boundary_condition} holds true also for the set $AS$, by \cite[Theorem 1]{HORNUNG2011}, there is a sequence of maps $(\tilde y_n) \subset W^{2,2}_{\text{iso}}(AS; \r^3) \cap C^\infty(\overline{AS}; \r^3)$ such that $\tilde y_n \to \tilde y$ in $W^{2,2}(AS; \r^3)$, where
    \[
    	W^{2, 2}_{\text{iso}}(AS; \r^3) = \{y \in W^{2,2}(AS; \r^3) \colon \nabla y^T \nabla y = \Id_2\} \,.
    \]
    Let $\tilde v_n = \tilde y_n \cdot e_3$. By \cite[Theorem 7]{FRIESECKE2006}, we deduce that $\det(\nabla^2 \tilde v_n) = 0$ for every $n \in \n$. Clearly $\tilde v_n \to \tilde v$ in $W^{2, 2}(AS)$, so $v_n(x) = |U^{-1}e_3| \tilde v_n(Ax) \to v$ in $W^{2,2}(S)$. Lastly, $\det(\nabla^2 v_n) = \det(|U^{-1}e_3|\nabla^2 \tilde v_n A^{-2}) = 0$. 
\end{proof}

%% file: acknowledgments.tex
\noindent \textbf{Acknowledgments.} The author acknowledges support by PRIN2022 number 2022J4FYNJ funded by MUR, Italy, and by the European Union--Next Generation EU. The author is a member of GNAMPA--INdAM.